\documentclass[11pt]{amsart}

\usepackage{amssymb}
\usepackage{bm}
\usepackage[centertags]{amsmath}
\usepackage{amsfonts}
\usepackage{amsthm}
\usepackage{mathrsfs}

\theoremstyle{definition}
\newtheorem{thm}{Theorem}[section]
\newtheorem{dfn}[thm]{Definition}
\newtheorem{lem}[thm]{Lemma}
\newtheorem{prp}[thm]{Proposition}
\newtheorem{cor}[thm]{Corollary}
\newtheorem{rmk}[thm]{Remark}

\newtheorem*{thm*}{Theorem}

\newtheorem*{cor*}{Corollary}

\newtheorem*{prp*}{Proposition}

\newtheorem*{rmk*}{Remark}

\newtheorem*{ntt}{Notation}

\newcommand{\inn}{\in\mathbb{N}}

\newcommand{\e}{\varepsilon}
\newcommand{\al}{\alpha}
\newcommand{\de}{\delta}
\newcommand{\la}{\lambda}
\newcommand{\La}{\Lambda}
\newcommand{\be}{\beta}
\newcommand{\adx}{\al\big(\{x_k\}_k\big)}
\newcommand{\bdx}{\be\big(\{x_k\}_k\big)}
\newcommand{\ady}{\al\big(\{y_k\}_k\big)}
\newcommand{\bdy}{\be\big(\{y_k\}_k\big)}
\newcommand{\adz}{\al\big(\{z_k\}_k\big)}
\newcommand{\bdz}{\be\big(\{z_k\}_k\big)}
\newcommand{\adw}{\al\big(\{w_k\}_k\big)}
\newcommand{\bdw}{\be\big(\{w_k\}_k\big)}
\newcommand{\Sn}{\mathcal{S}_n}
\newcommand{\X}{\mathfrak{X}_{_{^\text{ISP}}}}

\DeclareMathOperator{\supp}{supp}

\DeclareMathOperator{\ran}{ran}

\DeclareMathOperator{\scc}{succ}

\DeclareMathOperator{\sgn}{sgn}

\DeclareMathOperator{\dist}{dist}

\long\def\symbolfootnote[#1]#2{\begingroup%
\def\thefootnote{\fnsymbol{footnote}}\footnote[#1]{#2}\endgroup}

\begin{document}

\title [Hereditary Invariant Subspace Property]{A Reflexive HI space with the hereditary Invariant Subspace Property}
\dedicatory{Dedicated to the memory of Edward Odell}

\author[S.A. Argyros, P. Motakis]{Spiros A. Argyros and Pavlos
Motakis}
\address{National Technical University of Athens, Faculty of Applied Sciences,
Department of Mathematics, Zografou Campus, 157 80, Athens,
Greece} \email{sargyros@math.ntua.gr, pmotakis@central.ntua.gr}

\maketitle

\symbolfootnote[0]{\textit{2010 Mathematics Subject
Classification:} Primary 46B03, 46B06, 46B25, 46B45, 47A15}

\symbolfootnote[0]{\textit{Key words:} Spreading models, Strictly
singular operators, Invariant subspaces, Reflexive spaces,
Hereditarily indecomposable spaces}

\begin{abstract}
A separable Banach space $X$ satisfies the Invariant Subspace
Property (ISP) if every bounded linear operator
$T\in\mathcal{L}(X)$ admits a non trivial closed invariant
subspace. In this paper we present the first example of a
reflexive Banach space $\X$ satisfying the ISP. Moreover, this is
the first example of a Banach space satisfying the hereditary ISP,
namely every infinite dimensional subspace of it satisfies the
ISP. The space $\X$ is hereditarily indecomposable (HI) and every
operator $T\in\mathcal{L}(\X)$ is of the form $\la I + S$ with $S$
a strictly singular operator. The critical property of the
strictly singular operators of $\X$ is that the composition of any
three of them is a compact one. The construction of $\X$ is based
on saturation methods and it uses as an unconditional frame
Tsirelson space. The new ingredient in the definition of the space
is the saturation under constraints, a method initialized in a
fundamental work of Edward Odell and Thomas Schlumprecht.
\end{abstract}

\section*{Introduction}

The invariant subspace problem asks whether every bounded linear
operator on an infinite dimensional separable Banach space admits
a non-trivial closed invariant subspace. A classical result of M.
Aronszajn and K.T. Smith \cite{AS} asserts that the problem has a
positive answer for compact operators. This result was extended by
V. Lomonosov \cite{L} for operators on complex Banach spaces that
commute with a non-trivial compact operator. Recently G. Sirotkin
\cite{Sir} has presented a version of Lomonosov's theorem for real
spaces. It is also known that the problem, in its full generality,
has a negative answer. Indeed P. Enflo \cite{E} and subsequently
C. J. Read \cite{R1},\cite{R2} have provided several examples of
operators on non-reflexive Banach spaces that do not admit a
non-trivial invariant subspace. In particular, in a profound study concerning spaces admitting operators without non trivial invariant subspaces, C. J. Read has proven that every separable Banach space that contains either $c_0$ or a complemented subspace isomorphic to $\ell_1$ or $J_\infty$, admits an operator without non trivial closed invariant subspaces \cite{R4}. A comprehensive study of Read's methods of constructing operators with no non trivial invariant subspaces can be found in \cite{GR1}, \cite{GR2}. Also recently a non-reflexive
hereditarily indecomposable (HI) Banach space $\mathfrak{X}_K$
with the ``scalar plus compact'' property has been constructed
\cite{AH}. This is a $\mathcal{L}_\infty$ space with separable
dual, resulting from a combination of HI techniques with the
fundamental J. Bourgain and F. Delbaen construction \cite{BD}. As
consequence, the space $\mathfrak{X}_K$ satisfies the Invariant
Subspace Property (ISP). Moreover, recently a $\mathcal{L}_\infty$ space, containing $\ell_1$ isomorphically, with the scalar plus compact property, has been constructed \cite{AHR}. Let us point out, that the latter shows that Read's result concerning separable spaces containing $\ell_1$ as a complemented subspace, is not extended to separable spaces containing $\ell_1$. All the above results provide no
information in either direction within the class of reflexive
Banach spaces. The importance of a result in this class, concerning the Invariant Subspace Problem,
 is reflected in the concluding phrase of C. J. Read in \cite{R4} where the following is stated. ``It is clear that we cannot
go much further until and unless we solve the invariant subspace problem on a
reflexive Banach space.''

The aim of the present work is to construct the first example of a
reflexive Banach space $\X$ with the Invariant Subspace Property,
which is also the example of a Banach space with the hereditary
Invariant Subspace Property. This property is not proved for the
aforementioned space $\mathfrak{X}_K$. It is notable that no
subspace of $\X$ has the ``scalar plus compact'' property. More
precisely, the strictly singular operators\footnote[1]{A bounded
linear operator is called strictly singular, if its restriction on
any infinite dimensional subspace is not an isomorphism.} on every
subspace $Y$ of $\X$ form a non separable ideal (in particular,
the strictly singular non-compact are non-separable).

The space $\X$ is a hereditarily indecomposable space and every
operator $T\in\mathcal{L}(\X)$ is of the form $T = \la I + S$ with
$S$ strictly singular. We recall that there are strictly singular
operators in Banach spaces without non-trivial invariant subspaces
\cite{R3}. On the other hand, there are spaces where the ideal of
strictly singular operators does not coincide with the
corresponding one of compact operators and every strictly singular
operator admits a non-trivial invariant subspace. The most
classical spaces with this property are $L^p[0,1], 1\leqslant
p <\infty$ and  $C[0,1]$. This is a combination of Lomonosov-Sirotkin
theorem and the classical result, due to V. Milman \cite{M}, that the composition $TS$ is a compact
operator, for any $T,S$ strictly singular operators, on any of the
above spaces. In \cite{ABM}, Tsirelson like spaces satisfying
similar properties are presented. The possibility of constructing
a reflexive space with ISP without the ``scalar plus compact''
property emerged from an earlier version of \cite{ABM}.

The following describes the main properties of the space $\X$.

\begin{thm*} There exists a reflexive space $\X$ with a Schauder
basis $\{e_n\}_{n\inn}$ satisfying the following properties.
\begin{enumerate}

\item[(i)] The space $\X$ is hereditarily indecomposable.

\item[(ii)] Every seminormalized weakly null sequence
$\{x_n\}_{n\inn}$ has a subsequence generating either $\ell_1$ or
$c_0$ as a spreading model. Moreover every infinite dimensional
subspace $Y$ of $\X$ admits both $\ell_1$ and $c_0$ as spreading
models.

\item[(iii)] For every $Y$ infinite dimensional closed subspace of
$\X$ and every $T\in\mathcal{L}(Y,\X),\; T = \la I_{_{Y,\X}} + S$
with $S$ strictly singular.

\item[(iv)] For every $Y$ infinite dimensional subspace of $\X$
the ideal $\mathcal{S}(Y)$ of the strictly singular operators is
non separable.

\item[(v)] For every $Y$ subspace of $\X$ and every $Q,S,T$ in
$\mathcal{S}(Y)$ the operator $QST$ is compact. Hence for every
$T\in\mathcal{S}(Y)$ either $T^3 = 0$ or $T$ commutes with a non
zero compact operator.

\item[(vi)] For every $Y$ infinite dimensional closed subspace of
$X$ and every $T\in\mathcal{L}(Y)$, $T$ admits a non-trivial
closed invariant subspace. In particular every $T\neq\la I_Y$, for
$\la\in\mathbb{R}$ admits a non-trivial hyperinvariant subspace.

\end{enumerate}
\end{thm*}

It is not clear to us if the number of operators in property (v)
can be reduced. For defining the space $\X$ we use classical
ingredients like the coding function $\sigma$, the interaction
between conditional and unconditional structure, but also some new
ones which we are about to describe.

In all previous HI constructions, one had to use a mixed Tsirelson
space as the unconditional frame on which the HI norm is built.
Mixed Tsirelson spaces appeared with Th. Schlumprecht space
\cite{Sch}, twenty years after Tsirelson construction \cite{T}.
They became an inevitable ingredient for any HI construction,
starting with the W.T. Gowers and B. Maurey celebrated example
\cite{GM}, and followed by myriads of others \cite{AD2},\cite{AT}
etc. The most significant difference in the construction of $\X$
from the classical ones, is that it uses as an unconditional frame
the Tsirelson space itself.

As it is clear to the experts, HI constructions based on Tsirelson
space, are not possible if we deal with a complete saturation of
the norm. Thus the second ingredient involves saturation under
constraints. This method was introduced by E. Odell and Th.
Schlumprecht \cite{OS1},\cite{OS2} for defining heterogeneous
local structure in HI spaces, a method also used in \cite{ABM}. By
saturation under constraints we mean that the operations
$(\frac{1}{2^n},\Sn)$ (see Remark \ref{remxx1.5}) are applied on
very fast growing families of averages, which are either
$\al$-averages or $\be$-averages. The $\al$-averages have been
also used in \cite{OS1},\cite{OS2}, while $\be$-averages are
introduced to control the behaviour of special functionals. It is
notable that although the $\al,\be$-averages do not contribute to
the norm of the vectors in $\X$, they are able to neutralize the
action of the operations $(\frac{1}{2^n},\Sn)$ on certain
sequences and thus $c_0$ spreading models become abundant. This
significant property yields the structure of $\X$ described in the
above theorem.

Let us briefly describe some further structural properties of the
space $\X$.

The first and most crucial one is that for a $(n,\e)$ special
convex combination (see Definition \ref{defscc}) $\sum_{i\in
F}c_ix_i$, with $\{x_i\}_{i\in F}$ a finite normalized block
sequence, we have that
\begin{equation*}
\|\sum_{i\in F}c_ix_i\|\leqslant \frac{6}{2^n} + 12\e
\end{equation*}
This evaluation is due to the fact that the space is built on
Tsirelson space and differs from the classical asymptotic $\ell_1$
HI spaces (i.e. \cite{AD2},\cite{AT}) where seminormalized
$(n,\e)$ special convex combinations exist in every block
subspace. A consequence of the above, is that the frequency of the
appearance of RIS sequences is significantly increased, which
among others yields the following. Every strictly singular
operator maps sequences generating $c_0$ spreading models to norm
null ones. Furthermore, we classify weakly null sequences into
sequences of rank 0, namely norm null ones, sequences of rank 1,
namely sequences generating $c_0$ as a spreading model and
sequences of rank 2 or 3, namely sequences generating $\ell_1$ as
a spreading model. The main result concerning these ranks is the
following. If $Y$ is an infinite dimensional closed subspace of
$\X$ and $T$ is a strictly singular operator on $Y$, then it maps
sequences of non zero rank, to sequences of strictly smaller rank.
Combining the above properties we conclude property (v) of the
above theorem.\vskip5pt

We thank G. Costakis for bringing to our attention G. Sirotkin's
paper \cite{Sir}.

\section{The norming set of the space $\X$}

In this section we define the norming set $W$ of the space $\X$.
This set is defined with the use of the sequence $\{\Sn\}_n$ which
we remind below and also families of $\Sn$-admissible functionals.

As we have mentioned in the introduction, the set $W$ will be a
subset of the norming set $W_T$ of the Tsirelson space.

\subsection*{The Schreier families} The Schreier families is
an increasing sequence of families of finite subsets of the
naturals, first appeared in \cite{AA}, inductively defined in the
following manner.

Set $\mathcal{S}_0 = \big\{\{n\}: n\inn\big\}$ and $\mathcal{S}_1
= \{F\subset\mathbb{N}: \#F\leqslant\min F\}$.

Suppose that $\Sn$ has been defined and set $\mathcal{S}_{n+1} =
\{F\subset\mathbb{N}: F = \cup_{j = 1}^k F_j$, where $F_1 <\cdots<
F_k\in\Sn$ and $k\leqslant\min F_1\}$

If for $n,m\inn$ we set $\Sn*\mathcal{S}_m = \{F\subset\mathbb{N}:
F = \cup_{j = 1}^k F_j$, where $F_1 <\cdots< F_k\in\mathcal{S}_m$
and $\{\min F_j: j=1,\ldots,k\}\in\Sn\}$, then it is well known
that $\Sn*\mathcal{S}_m = \mathcal{S}_{n+m}$.

\begin{ntt} A sequence of vectors $x_1
<\cdots<x_k$ in $c_{00}$ is said to be $\Sn$-admissible if
$\{\min\supp x_i: i=1,\ldots,k\}\in\Sn$.

Let $G\subset c_{00}$. A vector $f\in G$ is said to be an average
of size $s(f) = n$, if there exist $f_1,\ldots,f_d\in G,
d\leqslant n$, such that $f = \frac{1}{n}(f_1+\cdots+f_d)$.

A sequence $\{f_j\}_j$ of averages in $G$ is said to be very fast
growing, if $f_1<f_2<\ldots$, $s(f_j)>2^{\max\supp f_{j-1}}$ and
$s(f_j) > s(f_{j-1})$ for $j>1$.
\end{ntt}

\subsection*{The coding function} Choose $L = \{\ell_k:
k\inn\}, \ell_1
>2$ an infinite subset of the naturals such that:
\begin{enumerate}

\item[(i)] For any
$k\inn$ we have that $\ell_{k+1} > 2^{2\ell_k}$ and

\item[(ii)] $\sum_{k=1}^\infty
\frac{1}{2^{\ell_k}}<\frac{1}{1000}$.

\end{enumerate}
Decompose $L$ into further infinite subsets $L_1, L_2$. Set
\begin{eqnarray*}
\mathcal{Q} &=&
\big\{\big((f_1,n_1),\ldots,(f_m, n_m)\big): m\inn, \{n_k\}_{k=1}^m\subset\mathbb{N}, f_1 <\ldots <f_m\in c_{00}\\
&&\text{with}\;f_k(i)\in\mathbb{Q},\;\text{for}\;i\inn,
k=1,\ldots,m\}
\end{eqnarray*}
Choose a one to one function $\sigma:\mathcal{Q}\rightarrow L_2$, called the coding function, such that
for any $\big((f_1,n_1),\ldots,(f_m, n_m)\big)\in\mathcal{Q}$, we have that
\begin{equation*}
\sigma\big((f_1, n_1),\ldots,(f_m,n_m)\big) > 2^{n_m}\cdot\max\supp f_m
\end{equation*}

\begin{rmk}
For any $n\inn$ we have that
$\#L\cap\{n,\ldots,2^{2n}\}\leqslant 1$. \label{remark1.1}
\end{rmk}

\subsection*{The norming set} The norming set $W$ is defined to
be the smallest subset of $c_{00}$ satisfying the following
properties:\vskip3pt

\noindent {\bf 1.} The set $\{\substack{+\\[-2pt]-} e_n\}_{n\inn}$
is a subset of $W$, for any $f\in W$ we have that $-f\in W$, for
any $f\in W$ and any $I$ interval of the naturals we have that
$If\in W$ and $W$ is closed under rational convex
combinations. Any $f = \substack{+\\[-2pt]-} e_n$ will be called a functional of type 0.\vskip3pt

\noindent {\bf 2.} The set $W$ contains any functional $f$ which
is of the form $f = \frac{1}{2^n}\sum_{j=1}^d\al_j$, where
$\{\al_j\}_{j=1}^d$ is an $\Sn$-admissible and very fast growing
sequence of $\al$-averages in $W$. If $I$ is an interval of the
naturals, then $g = \substack{+\\[-2pt]-}If$ is called a functional of type I$_\al$, of
weight $w(g) = n$.\vskip3pt

\noindent {\bf 3.} The set $W$ contains any functional $f$ which
is of the form $f = \frac{1}{2^n}\sum_{j=1}^d\be_j$, where
$\{\be_j\}_{j=1}^d$ is an $\Sn$-admissible and very fast growing
sequence of $\be$-averages in $W$. If $I$ is an interval of the
naturals, then $g = \substack{+\\[-2pt]-}If$ is called a functional of type I$_\be$, of
weight $w(g) = n$.\vskip3pt

\noindent {\bf 4.} The set $W$ contains any functional $f$ which
is of the form $f = \frac{1}{2}\sum_{j=1}^df_j$, where
$\{f_j\}_{j=1}^d$ is an $\mathcal{S}_1$-admissible special
sequence of type I$_\al$ functionals. This means that $w(f_1)\in
L_1$ and $w(f_j) =
\sigma\big(\big(f_1,w(f_1)\big),\ldots,\big(f_{j-1},w(f_{j-1})\big)\big)$,
for $j>1$. If $I$ is an interval of the naturals, then $g = \substack{+\\[-2pt]-}If$
is called a functional of type II with weights $\widehat{w}(g) =
\{w(f_j) : \ran f_j\cap I\neq\varnothing\}$.\vskip3pt

We call an $\al$-average any average $\al\in W$ of the form $\al =
\frac{1}{n}\sum_{j=1}^df_j, d\leqslant n$, where
$f_1<\cdots<f_d\in W$.

We call a $\be$-average any average $\be\in W$ of the form $\be
= \frac{1}{n}\sum_{j=1}^df_j, d\leqslant n$, where
$f_1,\ldots,f_d\in W$ are functionals of type II, with disjoint
weights $\widehat{w}(f_j)$.

In general, we call a convex combination any $f\in W$ that is not
of type 0, I$_\al$, I$_\be$ or II.\vskip3pt

For $x\in c_{00}$ define $\|x\| = \sup\{f(x): f\in W\}$ and $\X =
\overline{(c_{00}(\mathbb{N}),\|\cdot\|)}$. Evidently $\X$ has a
bimonotone basis.

One may also describe the norm on $\X$ with an implicit formula.
Indeed, for some $x\in\X$, we have that
\begin{equation*}
\|x\| =
\max\big\{\|x\|_0,\;\|x\|_{II},\;\sup\{\frac{1}{2^n}\sum_{j=1}^d\|E_jx\|_{k_j}^\al\},\;\sup\{\frac{1}{2^n}\sum_{j=1}^d\|E_jx\|_{k_j}^\be\}\big\}
\end{equation*}
where the inner suprema are taken over all $n\inn$, all
$\Sn$-admissible intervals $\{E_j\}_{j=1}^d$ of the naturals and
$k_1<\cdots<k_d$ such that $k_j > 2^{\max E_{j-1}}$ for $j>1$.

By $\|x\|_{II}$ we denote\\
$\|x\|_{II} = \sup\{f(x): f\in W$ is a functional of type II$\}$\\
whereas for $j\inn$, by $\|x\|_j^\al$ we denote\\
$\|x\|_j^\al = \sup\{\al(x): \al\in W$ is an $\al$-average of size
$s(\al) = j\}$

Similarly, by $\|x\|_j^\be$ we denote\\
$\|x\|_j^\be = \sup\{\be(x): \be\in W$ is a $\be$-average of size
$s(\be) = j\}$.

\begin{rmk}
Very fast growing sequences of $\al$-averages have been considered
by E. Odell and Th. Schlumprecht in \cite{OS1}, \cite{OS2} and
were also used in \cite{ABM}. However, $\be$-averages are a new
ingredient, introduced to control the behaviour of type II
functionals on block sequences. The $\be$-averages can also be
used to provide an alternative and simpler approach of the main
result in \cite{OS2}.

As we have mentioned in the introduction, the $\|z\|_j^\al,
\|z\|_j^\be$, which are averages, do not contribute to the norm of
the vector $z$. On the other hand, the $\{\|\cdot\|_j^\al\}_j,
\{\|\cdot\|_j^\be\}_j$ have a significant role for the structure
of the space $\X$.
\end{rmk}

\begin{rmk} The norming set $W$ can be inductively constructed to
be the union of an increasing sequence of subsets
$\{W_m\}_{m=0}^\infty$ of $c_{00}$, where $W_0 = \{\substack{+\\[-2pt]-}
e_n\}_{n\inn}$ and if $W_m$ has been constructed, then set
$W_{m+1}^\al$ to be the closure of $W_m$ under taking
$\al$-averages, $W_{m+1}^{\text{I}_\al}$ to be the closure of
$W_{m+1}^\al$ under taking type I$_\al$ functionals,
$W_{m+1}^{\text{I}_\be}$ to be the closure of
$W_{m+1}^{\text{I}_\al}$ under taking type I$_\be$ functionals,
$W_{m+1}^{\text{II}}$ to be the closure of
$W_{m+1}^{\text{I}_\be}$ under taking type II functionals,
$W_{m+1}^\be$ to be the closure of $W_{m+1}^{\text{II}}$ under
taking $\be$-averages and finally $W_{m+1}$ to be the closure of
$W_{m+1}^\be$ under taking rational convex combinations.
\label{remark1.2}
\end{rmk}

\subsection*{Tsirelson space} Tsirelson's initial definition
\cite{T} of the first Banach space not containing any $\ell_p,
1\leqslant p<\infty$ or $c_0$, concerned the dual of the so called
Tsirelson norm which was introduced by T. Figiel and W. B. Johnson
\cite{FJ} and satisfies the following implicit formula.
\begin{equation*}
\|x\|_T =
\max\big\{\|x\|_0,\;\sup\{\frac{1}{2}\sum_{j=1}^d\|E_jx\|_T\}\big\}
\end{equation*}
where $x\in c_{00}$ and the inner supremum is taken over all
successive subsets of the naturals $d \leqslant E_1 <\cdots <E_d$.
Tsirelson space $T$ is defined to be the completion of
$(c_{00},\|\cdot\|_T)$. In the sequel by Tsirelson norm and
Tsirelson space we will mean the norm and the corresponding space
from \cite{FJ}.

As is well known, a norming set $W_T$ of Tsirelson space is the
smallest subset of $c_{00}$ satisfying the following
properties.\vskip3pt

\noindent {\bf 1.} The set $\{\substack{+\\[-2pt]-} e_n\}_{n\inn}$
is a subset of $W_T$, for any $f\in W_T$ we have that $-f\in W_T$,
for any $f\in W_T$ and any $E$ subset of the naturals we have that
$Ef\in W_T$ and $W_T$ is closed under rational convex
combinations.\vskip3pt

\noindent {\bf 2.} The set $W_T$ contains any functional $f$ which
is of the form $f = \frac{1}{2}\sum_{j=1}^df_j$, where
$\{f_j\}_{j=1}^d$ is a $\mathcal{S}_1$ admissible sequence in
$W_T$.

\begin{rmk} The following are well known facts about Tsirelson
space.\label{remark1.3}
\begin{enumerate}

\item[(i)] The norming set $W_T$ can be inductively constructed to
be the union of an increasing sequence of subsets
$\{W_T^m\}_{m=0}^\infty$ of $c_{00}$, in a similar manner as
above.

\item[(ii)] The set $W_T^\prime$, which is the smallest subset of
$c_{00}$ satisfying the following properties, also is a norming
set for Tsirelson space.\vskip3pt

\noindent {\bf 1.} The set $\{\substack{+\\[-2pt]-} e_n\}_{n\inn}$
is a subset of $W_T^\prime$, for any $f\in W_T^\prime$ we have
that $-f\in W_T^\prime$ and for any $f\in W_T^\prime$ and any $E$
subset of the naturals we have that $Ef\in W_T^\prime$.\vskip3pt

\noindent {\bf 2.} The set $W_T^\prime$ contains any functional
$f$ which is of the form $f = \frac{1}{2}\sum_{j=1}^df_j$, where
$\{f_j\}_{j=1}^d$ is a $\mathcal{S}_1$ admissible sequence in
$W_T^\prime$.

\end{enumerate}
\end{rmk}

\begin{rmk}It is easy to check that the norming set $W_T$ of
Tsirelson space is closed under $(\frac{1}{2^n},\Sn)$ operations,
namely for any $f_1<\cdots<f_d$ in $W_T$\; $\Sn$-admissible, the
functional $\frac{1}{2^n}\sum_{j=1}^df_j\in W_T$. This explains
that the norming set $W$ of the space $\X$ is a subset of $W_T$.
Therefore Tsirelson space is the unconditional frame on which the
norm of $\X$ is built. As we mentioned in the introduction, $\X$
is the first HI construction which uses Tsirelson space instead of
a mixed Tsirelson one.\label{remxx1.5}
\end{rmk}

As it is shown in \cite{CJT} (see also \cite{CS}), an equivalent
norm on Tsirelson space is described by the following implicit
formula. For $x\in c_{00}$ set
\begin{equation*}
|\!|\!|x|\!|\!| =
\max\big\{\|x\|_0,\;\sup\{\frac{1}{2}\sum_{j=1}^{2d}|\!|\!|E_jx|\!|\!|\}\big\}
\end{equation*}
where the inner supremum is taken over all successive subsets of
the naturals $d \leqslant E_1 <\cdots <E_{2d}$. Then, for any
$\{c_k\}_{k=1}^n\subset\mathbb{R}$, the following holds.

\begin{equation}
\|\sum_{k=1}^nc_ke_k\|_T \leqslant
|\!|\!|\sum_{k=1}^nc_ke_k|\!|\!| \leqslant
3\|\sum_{k=1}^nc_ke_k\|_T \label{tsirelsonequivalence}
\end{equation}

\begin{rmk}
A norming set $W_{(T,|\!|\!|\cdot|\!|\!|)}$ for
$(T,|\!|\!|\cdot|\!|\!|)$ is also defined in a similar manner as
$W_T$.\label{remarknormingthreebars}
\end{rmk}

\subsection*{Special convex combinations} Next, we remind the notion
of the $(n,\e)$ special convex combinations, (see
\cite{AD2},\cite{AGR},\cite{AT}) which is one of the main tools,
used in the sequel.

\begin{dfn} Let $x = \sum_{k\in F}c_ke_k$ be a vector in $c_{00}$.
Then $x$ is said to be a $(n,\e)$ basic special convex combination
(or a $(n,\e)$ basic s.c.c.) if:

\begin{enumerate}

\item[(i)] $F\in\Sn, c_k\geqslant 0$, for $k\in F$ and $\sum_{k\in
F}c_k = 1$.

\item[(ii)] For any $G\subset F, G\in\mathcal{S}_{n-1}$, we have
that $\sum_{k\in G}c_k < \e$.

\end{enumerate}

\end{dfn}

The next result is from \cite{AMT}. For a proof see \cite{AT},
Chapter 2, Proposition 2.3.

\begin{prp}
For any $M$ infinite subset of the naturals, any $n\inn$ and
$\e>0$, there exists $F\subset M, \{c_k\}_{k\in F}$, such that $x
= \sum_{k\in F}c_ke_k$ is a $(n,\e)$ basic s.c.c. \label{prop1.5}
\end{prp}

\begin{dfn}
Let $x_1 <\cdots<x_m$ be vectors in $c_{00}$ and $\psi(k) =
\min\supp x_k$, for $k=1,\ldots,m$. Then $x = \sum_{k=1}^mc_kx_k$
is said to be a $(n,\e)$ special convex combination (or $(n,\e)$
s.c.c.), if $\sum_{k=1}^mc_ke_{\psi(k)}$ is a $(n,\e)$ basic
s.c.c.\label{defscc}
\end{dfn}

\section{Basic evaluations for special convex combinations}

In this section we prove the basic inequality for block sequences
in $\X$, with the auxiliary space actually being Tsirelson space.
This will allow us to evaluate the norm of $(n,\e)$ special convex
combinations and it is critical throughout the rest of the paper.

\begin{dfn} Let $f\in W$ be a functional of type I$_\al$ or I$_\be$, of weight
$w(f) = n$, $f = \frac{1}{2^n}\sum_{j=1}^df_j$. Then, by
definition, there exist $F_1<\cdots<F_p$ successive intervals of
the naturals such that:

\begin{enumerate}

\item[(i)]
 $\cup_{i=1}^pF_i = \{1,\ldots,d\}$

\item[(ii)]
$\{\min\supp f_j: j\in F_i\}\in\mathcal{S}_{n-1}$, for $i=1,\ldots,p$

\item[(iii)]
$\{\min\supp f_{\min F_i}: i=1,\ldots,p\}\in\mathcal{S}_1$

\end{enumerate}

Set $g_i = \frac{1}{2^{n-1}}\sum_{j\in F_i}f_j$, for
$i=1,\ldots,p$. We call $\{g_i\}_{i=1}^p$ a Tsirelson analysis of
$f$.
\end{dfn}

\begin{rmk} If $f\in W$ is a functional of type I$_\al$ or I$_\be$ and
$\{f_i\}_{i=1}^p$ is a Tsirelson analysis of $f$, then $f_i\in W$,
$\{f_i\}_{i=1}^p$ is $\mathcal{S}_1$-admissible and $f =
\frac{1}{2}\sum_{i=1}^pf_i$, although $\{f_i\}_{i=1}^p$ may not be
a very fast growing sequence of $\al$-averages or $\be$-averages.
Moreover, if $w(f)>1$, then $f_i$ is of the same type as $f$ and
$w(f_i) = w(f) - 1$ for $i=1,\ldots,p$. \label{remark2.3}
\end{rmk}

\subsection*{The tree analysis of a functional $\mathbf{f\in W}$} A
key ingredient for evaluating the norm of vectors in $\X$ is the
analysis of the elements $f$ of the norming set $W$. This is
similar to the corresponding concept that has occurred in almost
all previous HI and related constructions (i.e. \cite{AD1},
\cite{AD2}, \cite{AH}, \cite{AT}). Next we briefly describe the
tree analysis in our context.

For any functional $f\in W$ we associate a family
$\{f_\la\}_{\la\in\La}$, where $\La$ is a finite tree which is
inductively defined as follows.

Set $f_\varnothing =f$, where $\varnothing$ denotes the root of
the tree to be constructed. If $f$ is of type 0, then the tree
analysis of $f$ is $\{f_\varnothing\}$. Otherwise, suppose that
the nodes of the tree and the corresponding functionals have been
chosen up to a height $p$ and let $\la$ be a node of height $|\la|
= p$. If $f_\la$ is of type 0, then don't extend any further and
$\la$ is a maximal node of the tree.

If $f_\la$ is of type I$_\al$ or I$_\be$, set the immediate
successors of $\la$ to be the elements of the Tsirelson analysis
of $f_\la$.

If $f_\la$ is of type II, $f = \frac{1}{2}\sum_{j=1}^df_j$, set
the immediate successors of $\la$ to be the $\{f_j\}_{j=1}^d$.

If $f_\la$ is a convex combination, which includes $\al$-averages
and $\be$-averages, $f_\la = \sum_{j=1}^dc_jf_j$, set the
immediate successors of $\la$ to be the $\{f_j\}_{j=1}^d$.

By Remark \ref{remark1.2} it follows that the inductive
construction ends in finitely many steps and that the tree $\La$
is finite.

\begin{rmk} Let $f\in W$ and $\{f_\la\}_{\la\in\La}$ be a tree analysis of
$f$. Then for any $\la\in\La$ not a maximal node, such that
$f_\la$ is not a convex combination, we have that $f_\la =
\frac{1}{2}\sum_{\mu\in\scc(\la)}f_\mu$, where
$\{f_\mu\}_{\mu\in\scc(\la)}$ are $\mathcal{S}_1$-admissible and
by $\scc(\la)$ we denote the immediate successors of $\la$ in
$\La$. \label{remark2.4}
\end{rmk}

\begin{rmk} In a similar manner, for any $f\in W_T^\prime$ (see Remark \ref{remark1.3}\;(ii)), the
tree analysis of $f$ is defined.
\end{rmk}

\begin{prp} Let $x = \sum_{k\in F}c_ke_k$ be a $(n,\e)$ basic
s.c.c. and $G\subset F$. Then the following holds.

\begin{equation*}
\|\sum_{k\in G}c_ke_k\|_T \leqslant\frac{1}{2^n}\sum_{k\in G}c_k +
\e
\end{equation*}\label{prop2.1}
\end{prp}

\begin{proof}

Let $f\in W_T^\prime$. We may assume that $\supp f\subset G$. Set
$G_1 = \{k\in \supp f: |f(e_k)|\leqslant\frac{1}{2^n}\}, G_2 =
\supp f\setminus G_1$. Then clearly $|G_1f(\sum_{k\in
G}c_ke_k)|\leqslant \frac{1}{2^n}\sum_{k\in G}c_k$.

We will show by induction that $G_2\in\mathcal{S}_{n-1}$. Let
$\{f_\la\}_{\la\in\La}$ be a tree analysis of $G_2f$. Then it is
easy to see that $h(\La)\leqslant n-1$. For $\la$ a maximal node
in $\La$, we have that $\supp f_\la\in\mathcal{S}_0$. Assume that
for any $\la\in\La, |\la| = k>0$ we have that $\supp
f_\la\in\mathcal{S}_{n-1-k}$ and let $\la\in\La$, such that $|\la|
= k-1$. Then $f_\la = \frac{1}{2}\sum_{j=1}^df_{\la_j}$, where
$|\la_j| = k$, $\supp f_{\la_j}\in\mathcal{S}_{n-k-1}$ for
$j=1,\ldots,d$ and $\{\min\supp
f_{\la_j}:j=1,\ldots,d\}\in\mathcal{S}_1$. Then $\supp f_\la =
\cup_{j=1}^d\supp f_{\la_j}\in \mathcal{S}_{n-1-(k-1)}$.

The induction is complete and it follows that $G_2 = \supp G_2f\in
\mathcal{S}_{n-1}$ and therefore $G_2f(\sum_{k\in
G}c_ke_k)\leqslant \sum_{k\in G_2}c_k <\e$. Hence, $|f(\sum_{k\in
G}c_ke_k)| < \frac{1}{2^n}\sum_{k\in G}c_k + \e$.

\end{proof}

\begin{prp}[Basic Inequality] Let $\{x_k\}_k$ be a block sequence in $\X$ such that
$\|x_k\|\leqslant 1$, for all $k$ and let $f\in W$. Set $\phi(k) =
\max\supp x_k$, for all $k$. Then there exists $g\in
W_{(T,|\!|\!|\cdot|\!|\!|)}$ (see Remark
\ref{remarknormingthreebars}) such that $2g(e_{\phi(k)}) \geqslant
f(x_k)$, for all $k$. \label{basic}
\end{prp}

\begin{proof}
Let $\{f_\la\}_{\la\in\La}$ be a tree analysis of $f$. We will
inductively construct $\{g_\la\}_{\la\in\La}$ such that for any
$\la\in\La$ the following are satisfied.
\begin{enumerate}

\item[(i)] $g_\la\in W_{(T,|\!|\!|\cdot|\!|\!|)}$ and
$2g_\la(e_{\phi(k)}) \geqslant f_\la(x_k)$, for any $k$.

\item[(ii)] $\supp g_\la \subset \{\phi(k): \ran f_\la\cap\ran
x_k\neq\varnothing\}$

\end{enumerate}

For $\la\in\La$ a maximal node, if there exists $k$ such that
$\ran f_\la\cap\ran x_k\neq\varnothing$, set $g_\la =
e^*_{\phi(k)}$. Otherwise set $g_\la = 0$.

Let $\la\in\La$ be a non-maximal node, and suppose that
$\{g_\mu\}_{\mu>\la}$ have been chosen. We distinguish two
cases.\vskip3pt

\noindent {\em Case 1:} $f_\la$ is a convex combination (i.e.
$f_\la$ is not of type 0, I$_\al$, I$_\be$, or II).

If $f_\la = \sum_{\mu\in\scc(\la)}c_\mu f_\mu$, set $g_\la =
\sum_{\mu\in\scc(\la)}c_\mu g_\mu$.\vskip3pt

\noindent {\em Case 2:} $f_\la$ is not a convex combination.

If $f_\la = \frac{1}{2}\sum_{j=1}^df_{\mu_j}$, where $\scc(\la) =
\{\mu_j\}_{j=1}^d$ such that $f_{\mu_1}<\cdots<f_{\mu_d}$, set
\begin{eqnarray*}
G_\la &=&\{k: \ran f_\la\cap\ran x_k\neq\varnothing\}\\
G_1 &=& \{k\in G_\la:\;\text{there exists at most
one}\;j\;\text{with}\;\ran f_{\mu_j}\cap\ran
x_k\neq\varnothing\}\\
G_2 &=& \{k\in G_\la:\;\text{there exist at least
two}\;j\;\text{with}\;\ran f_{\mu_j}\cap\ran
x_k\neq\varnothing\}\\
I_j &=& \{k\in G_1: \ran x_k\cap\ran
f_{\mu_j}\neq\varnothing\}\quad\text{for}\;j = 1,\ldots,d.
\end{eqnarray*}
Observe that $\#G_2\leqslant d-1$.

For $j = 1,\ldots,d$ set $g_j^\prime = g_{\mu_j}|_{\phi(I_j)}$ and
for $k\in G_2$ set $g_k = e^*_{\phi(k)}$. It is easy to check that
if we set $g_\la = \frac{1}{2}\big(\sum_{j=1}^dg_j^\prime +
\sum_{k\in G_2}g_k\big)$, then $g_\la$ is the desired functional.

The induction is complete. Set $g = g_\varnothing$

\end{proof}

\begin{rmk}
In the previous constructions (see \cite{AD1}, \cite{AD2},
\cite{AH}, \cite{AT}), the basic inequality is used for estimating
the norm of linear combinations of block vectors which are RIS. In the present
paper the basic inequality is stronger, as it is able to provide upper estimations for any block vectors. Moreover, RIS sequences are defined in a different manner as in previous constructions and they also play a different role, which will be discussed in the sequel.
\end{rmk}

\begin{cor} Let $\{x_k\}_{k}$ be a block sequence in $\X$ such that
$\|x_k\|\leqslant 1, \{c_k\}_k\subset\mathbb{R}$ and $\phi(k) =
\max\supp x_k$ for all $k$. Then:
\begin{equation*}
\|\sum_kc_kx_k\| \leqslant 6\|\sum_kc_ke_{\phi(k)}\|_T
\end{equation*}\label{cor2.20}
\end{cor}

\begin{proof}
Let $f\in W$. Apply the basic inequality and take $g\in
W_{(T,|\!|\!|\cdot|\!|\!|)}$, such that if $\phi(k) = \max\supp
x_k$ and $y_k = \sgn(c_k)x_k$ for all $k$, we have that
$2g(e_{\phi(k)}) \geqslant f(y_k),\;\text{for any}\;k$. It follows
that
\begin{equation*}
2g(\sum_k|c_k|e_{\phi(k)})\geqslant f(\sum_kc_kx_k).
\end{equation*}
Therefore, applying \eqref{tsirelsonequivalence}, we get
\begin{equation*}
\|\sum_kc_kx_k\| \leqslant 2|\!|\!|\sum_k|c_k|e_{\phi(k)}|\!|\!| =
2|\!|\!|\sum_kc_ke_{\phi(k)}|\!|\!|\leqslant 2\cdot
3\|\sum_kc_ke_{\phi(k)}\|_T
\end{equation*}
\end{proof}

\begin{cor} Let $x = \sum_{k=1}^mc_kx_k$ be a $(n,\e)$ s.c.c. in $\X$, such
that $\|x_k\|\leqslant 1$, for $k=1,\ldots,m$. If
$F\subset\{1,\ldots,m\}$, then
\begin{equation*}
\|\sum_{k\in F}c_kx_k\| \leqslant \frac{6}{2^n}\sum_{k\in F}c_k +
12\e.
\end{equation*}
In particular, we have that $\|x\|\leqslant \frac{6}{2^n} +
12\e$.\label{cor2.21}
\end{cor}

\begin{proof}
Set $\phi(k) = \max\supp x_k, \psi(k) = \min\supp x_k$. Corollary
\ref{cor2.20} yields that $\|\sum_{k\in F}c_kx_k\| \leqslant
6\|\sum_{k\in F}c_ke_{\phi(k)}\|_T$.

Since, according to the assumption, $\sum_{k\in F}c_ke_{\psi(k)}$
is a $(n,\e)$ basic s.c.c., it easily follows that $\sum_{k\in
F}c_ke_{\phi(k)}$ is a $(n,2\e)$ basic s.c.c.

By Proposition \ref{prop2.1} the result follows.

\end{proof}

\begin{cor} The basis of $\X$ is shrinking.\label{cor2.22}
\end{cor}

\begin{proof}
Suppose that it is not. Then there exist $x^*\in\X^*, \|x^*\| =
1$, a normalized block sequence $\{x_k\}_{k\inn}$ in $\X$ and
$\de>0$, such that $x^*(x_k)>\de$, for all $k\inn$.

Choose $n\inn$, such that $\frac{1}{2^n}<\frac{\de}{12}$ and
$\e>0$, such that $\e<\frac{\de}{24}$. By Proposition
\ref{prop1.5} there exists $F$ a subset of $\mathbb{N}$, such that
$x = \sum_{k\in F}c_kx_k$ is a $(n,\e)$ s.c.c.

By Corollary \ref{cor2.21} we have that $\de > \|x\| \geqslant
x^*(x)
>\de$. A contradiction, which completes the proof.
\end{proof}

\begin{prp} The basis of $\X$ is boundedly complete.\label{prop2.23}
\end{prp}

\begin{proof}
Assume that it is not. Then there exist $\e>0$ and
$\{x_k\}_{k\inn}$ a block sequence in $\X$, such that $\|x_k\| >
\e$ and $\|\sum_{k=\ell}^{\ell+m}x_k\|\leqslant 1$, for all
$\ell,m\inn$.

Choose $k_0$ such that $d = \min\supp x_{k_0} > \frac{2}{\e}$. Set
$F_1 = \{k_0\}$ and inductively choose $F_1,\ldots,F_d$, intervals
of the naturals such that
\begin{enumerate}

\item[(i)] $\max F_j + 1 = \min F_{j+1}$, for $j<d$ and

\item[(ii)] $\#F_j > \max\{\#F_{j-1}, 2^{\max\supp x_{\max F_{j-1}}}\}$, for $1<j\leqslant d$.

\end{enumerate}

Then, if we set $y_j = \sum_{k\in F_j}x_k$, we have that $\|\sum_{j=1}^dy_j\|\leqslant 1$.

On the other hand, notice that for $j=1,\ldots,d$, there exists $\al_j$ an $\al$-average in $W$, such that
\begin{enumerate}

\item[(i)] $\ran \al_j\subset \ran y_j$, therefore $\{\al_j\}_{j=1}^d$ is $\mathcal{S}_1$-admissible.

\item[(ii)] $s(\al_j) = \#F_j$, therefore $\{\al_j\}_{j=1}^d$ is very fast growing.

\item[(iii)] $\al_j(y_j) > \e$
\end{enumerate}
From the above it follows $f = \frac{1}{2}\sum_{j=1}^d\al_j$ is a functional of type I$_\al$ in $W$ and $f(\sum_{j=1}^dy_j) > \frac{\e\cdot d}{2} > 1$. Since this cannot be the case, the proof is complete.
\end{proof}

These last two results and a well known result due to R. C. James
\cite{J}, allow us to conclude the following.

\begin{cor} The space $\X$ is reflexive.\label{cor2.24}
\end{cor}

\begin{dfn} Let $F$ either be $\mathbb{N}$ or an initial segment of the natural numbers. A block sequence $\{x_k\}_{k\in F}$ is said to be a
$(C,\{n_k\}_{k\in F})$\;$\al$-rapidly increasing sequence (or
$(C,\{n_k\}_{k\in F})$\;$\al$-RIS), for a positive constant $C\geqslant
1$ and a strictly increasing sequence of naturals $\{n_k\}_{k\in F}$, if
$\|x_k\| \leqslant C$ for all $k\in F$ and the following conditions are
satisfied.

\begin{enumerate}

\item[(i)] For any $k\in F$, for any functional $f$ of type I$_\al$ of
weight $w(f) = j < n_k$ we have that $|f(x_k)| < \frac{C}{2^j}$

\item[(ii)] For any $k\in F$ we have that
$\frac{1}{2^{n_{k+1}}}\max\supp x_{k} < \frac{1}{2^{n_k}}$

\end{enumerate}\label{def3.10}
\end{dfn}

\begin{rmk} Let $\{x_k\}_{k\inn}$ be a block sequence in $\X$, such that there
exist a positive constant $C$ and $\{n_k\}_{k\inn}$ strictly
increasing naturals, such that $\|x_k\|\leqslant C$ for all $k$
and condition (i) from Definition \ref{def3.10} is satisfied. Then
passing, if necessary, to a subsequence, $\{x_k\}_{k\inn}$ is
$(C,\{n_k\}_{k\inn})$\;$\al$-RIS.\label{rem3.11}
\end{rmk}

\begin{dfn}
Let $n\inn, C\geqslant 1, \theta>0$. A vector $x\in\X$ is called a $(C,\theta,n)$ vector if the following hold. There exist $0<\e<\frac{1}{36C2^{3n}}$ and $\{x_k\}_{k=1}^m$ a block sequence in $\X$ with $\|x_k\|\leqslant C$ for $k=1,\ldots,m$ such that

\begin{itemize}

\item[(i)] $\min\supp x_1 \geqslant 8C2^{2n}$

\item[(ii)] There exist $\{c_k\}_{k=1}^m\subset [0,1]$ such that $\sum_{k=1}^mc_kx_k$ is a $(n,\e)$ s.c.c.

\item[(iii)] $x = 2^n\sum_{k=1}^mc_kx_k$ and $\|x\| \geqslant \theta$.

\end{itemize}
If moreover there exist $\{n_k\}_{k=1}^m$ strictly increasing natural numbers with $n_1 > 2^{2n}$ such that $\{x_k\}_{k=1}^m$ is $(C,\{n_k\}_{k=1}^m)$ $\al$-RIS, then $x$ is called a $(C,\theta,n)$ exact vector.\label{defvector}

\end{dfn}

\begin{rmk}
Let $x$ be a $(C,\theta,n)$ vector in $\X$. Then, using Corollary \ref{cor2.21} we conclude that $\|x\| < 7C$.
\label{remexactest}
\end{rmk}

\section{The $\al,\be$ indices}

To each block sequence we will associate two indices related to
$\al$ and $\be$ averages. In this section we will show that every
normalized block sequence $\{x_n\}_n$ has a further normalized
block sequence $\{y_n\}_n$ such that on it both indices $\al$ and
$\be$ are equal to zero. As we will show in the next section, this
is sufficient, for a sequence to have a subsequence generating a
$c_0$ spreading model.

\begin{dfn}
Let $\{x_k\}_{k\inn}$ be a block sequence in $\X$ that satisfies
the following. For any $n\inn$, for any very fast growing sequence
$\{\al_q\}_{q\inn}$ of $\al$-averages in $W$ and for any
$\{F_k\}_{k\inn}$ increasing sequence of subsets of the naturals,
such that $\{\al_q\}_{q\in F_k}$ is $\Sn$-admissible, the
following holds. For any $\{x_{n_k}\}_{k\inn}$ subsequence of
$\{x_k\}_{k\inn}$, we have that $\lim_k\sum_{q\in
F_k}|\al_q(x_{n_k})| = 0$.

Then we say that the $\al$-index of $\{x_k\}_{k\inn}$ is zero and
write $\adx = 0$. Otherwise we write $\adx > 0$.
\end{dfn}

\begin{dfn}
Let $\{x_k\}_{k\inn}$ be a block sequence in $\X$ that satisfies
the following. For any $n\inn$, for any very fast growing sequence
$\{\be_q\}_{q\inn}$ of $\be$-averages in $W$ and for any
$\{F_k\}_{k\inn}$ increasing sequence of subsets of the naturals,
such that $\{\be_q\}_{q\in F_k}$ is $\Sn$-admissible, the
following holds. For any $\{x_{n_k}\}_{k\inn}$ subsequence of
$\{x_k\}_{k\inn}$, we have that $\lim_k\sum_{q\in
F_k}|\be_q(x_{n_k})| = 0$.

Then we say that the $\be$-index of $\{x_k\}_{k\inn}$ is zero and
write $\bdx = 0$. Otherwise we write $\bdx > 0$.
\end{dfn}

\begin{prp} Let $\{x_k\}_{k\inn}$ be a block sequence in $\X$. Then the
following assertions are equivalent.
\begin{enumerate}

\item[(i)] $\adx = 0$

\item[(ii)] For any $\e>0$ there exists $j_0\inn$ such that for
any $j\geqslant j_0$ there exists $k_j\inn$ such that for any
$k\geqslant k_j$, and for any $\{\al_q\}_{q=1}^d$
$\mathcal{S}_j$-admissible and very fast growing sequence of
$\al$-averages such that $s(\al_q) > j_0$, for $q=1,\ldots,d$, we
have that $\sum_{q=1}^d|\al_q(x_k)| < \e$.

\end{enumerate}\label{prop3.3}
\end{prp}

\begin{proof}
It is easy to prove that (i) follows from (ii), therefore we shall
only prove the inverse. Suppose that (i) is true and (ii) is not.

Then there exists $\e>0$ such that for any $j_0\inn$ there exists
$j\geqslant j_0$, such that for any $k_0\inn$, there exists
$k\geqslant k_0$ and $\{\al_q\}_{q=1}^d$ a
$\mathcal{S}_j$-admissible and very fast growing sequence of
$\al$-averages with $s(\al_q) > j_0$, for $q=1,\ldots,d$, such
that $\sum_{q=1}^d|\al_q(x_k)| \geqslant \e$.

We will inductively choose a subsequence $\{x_{n_i}\}_{i\inn}$ and
$\{\al^i\}_{i\inn}$ a very fast growing sequence of
$\al$-averages, such that $|\al^i(x_{n_i})| > \frac{\e}{2}$, for
any $i$. This evidently yields a contradiction.

For $j_0 = 1$, there exists $j_1\geqslant 1$, such that there
exists a subsequence $\{x_{k_j}\}_{j\inn}$ of $\{x_k\}_{k\inn}$, a
sequence $\{\al_q\}_{q\inn}$ of $\al$-averages with $s(\al_q) > 1$
for all $q\inn$ and $\{F_j\}_{j\inn}$ a sequence of increasing
intervals of the naturals, such that:
\begin{enumerate}

\item[(i)] $\{\al_q\}_{q\in F_j}$ is very fast growing and
$\mathcal{S}_{j_1}$-admissible.

\item[(ii)] $\sum_{q\in F_j}|\al_q(x_{k_j})|\geqslant \e$.

\item[(iii)] If $F_j^\prime = F_j\setminus\{\min F_j\}$, then
$\{\al_q\}_{q\in\cup_j F_j^\prime}$ is very fast growing.
\end{enumerate}
Since $\adx = 0$, we have that $\lim_j\sum_{q\in
F_j^\prime}|\al_q(x_{k_j})| = 0$. Choose $j$ such that $|\al_{\min
F_j}(x_{k_j})| >\frac{\e}{2}$ and set $n_1 = k_j, \al^1 =
\al_{\min F_j}$.

Suppose that we have chosen $n_1 <\cdots <n_p$ and
$\{a^i\}_{i=1}^p$ a very fast growing sequence of $\al$-averages,
such that $|\al^i(x_{n_i})| > \frac{\e}{2}$, for $i=1,\ldots,p$.

Set $j_0 = \max\{s(\al^p),\;2^{\max\supp\al^p}\}$ and repeat the
first inductive step to find an $\al$-average $\al$ with
$s(\al)>j_0$ and $x_k>x_{n_p}$, $x_k > \al^p$, such that
$|\al(x_k)|\geqslant\frac{\e}{2}$. Set $x_{n_{p+1}} = x_k$ and
$\al^{p+1} = \al|_{\ran x_k}$. The inductive construction is
complete and so is the proof.
\end{proof}

The proof of the next proposition is identical to the proof of the
previous one.

\begin{prp} Let $\{x_k\}_{k\inn}$ be a block sequence in $\X$. Then the
following assertions are equivalent.
\begin{enumerate}

\item[(i)] $\bdx = 0$

\item[(ii)] For any $\e>0$ there exists $j_0\inn$ such that for
any $j\geqslant j_0$ there exists $k_j\inn$ such that for any
$k\geqslant k_j$, and for any $\{\be_q\}_{q=1}^d$
$\mathcal{S}_j$-admissible and very fast growing sequence of
$\be$-averages such that $s(\be_q) > j_0$, for $q=1,\ldots,d$, we
have that $\sum_{q=1}^d|\be_q(x_k)| < \e$.

\end{enumerate}\label{prop3.4}
\end{prp}

\begin{prp}
Let $\{x_k\}_{k\inn}$ be a seminormalized block sequence in $\X$,
such that either $\adx > 0$, or $\bdx > 0$.

Then there exists
$\theta>0$ and a subsequence $\{x_{n_k}\}_{k\inn}$ of
$\{x_k\}_{k\inn}$, that generates an $\ell_1^n$ spreading model,
with a lower constant $\frac{\theta}{2^n}$, for all $n\inn$.

More precisely, for every $n\inn$ and $F\subset\mathbb{N}$ with $\{\min\supp x_{n_k}: k\in F\}\in\mathcal{S}_n$ and $\{c_k\}_k\subset\mathbb{R}$, we have that $\|\sum_{k\in F}c_kx_{n_k}\| \geqslant \frac{\theta}{2^n}\sum_{k\in F}|c_k|$.

In particular, for any $k_0,n\inn$, there exists $F$ a
finite subset of $\mathbb{N}$ with $\min F\geqslant k_0$ and
$\{c_k\}_{k\in F}$, such that $x = 2^n\sum_{k\in F}c_kx_{n_k}$ is a
$(C,\theta,n)$ vector, where $C = \sup\{\|x_k\|: k\inn\}$.

If moreover $\{x_k\}_k$ is $(C^\prime,\{n_k\}_k)$ $\al$-RIS, then $x$ can be chosen to be a $(C^{\prime\prime},\theta,n)$ exact vector, where $C^{\prime\prime} = \max\{C,C^\prime\}$.
\label{prop3.5}
\end{prp}

\begin{proof}
Assume that $\adx > 0$. Then there exist $\ell\inn, \e>0,
\{\al_q\}_{q\inn}$ a very fast growing sequence of $\al$-averages,
$\{F_k\}_{k\inn}$ increasing subsets of the naturals such that
$\{\al_q\}_{q\in F_k}$ is $\mathcal{S}_\ell$-admissible for all
$k\inn$ and $\{x_{n_k}\}_{k\inn}$ a subsequence of
$\{x_k\}_{k\inn}$, such that $\sum_{q\in F_k}|\al_q(x_{n_k})| >
\e$, for all $k\inn$. Pass, if necessary ,to a subsequence, again
denoted by $\{x_{n_k}\}_{k\inn}$, generating some spreading model.

By changing the signs and restricting the ranges of the $\al_q$,
we may assume that $\sum_{q\in F_k}\al_q(x_{n_k}) > \e$, for all
$k\inn$ and $\ran \al_q\subset \ran x_{n_k}$ for any $q\in F_k$
and $k\inn$. Set $\theta = \frac{\e}{2^\ell}$.

Let $k_0,n\inn$ and choose $0<\eta<\frac{1}{36C2^{3n}}$. By Proposition \ref{prop1.5} there exists
$F$ a finite subset of $\{n_k: k\geqslant \max\{k_0, 8C2^{2n}\}\}$ and
$\{c_k\}_{k\in F}$, such that $x^\prime = \sum_{k\in F}c_kx_{n_k}$ is a
$(n,\eta)$ s.c.c.

Set $f = \frac{1}{2^{\ell + n}}\sum_{k\in F}\sum_{q\in
F_{n_k}}\al_q$. Then $f$ is a functional of type I$_\al$ in $W$
and $f(x^\prime) > \frac{\e}{2^{\ell+n}} = \frac{\theta}{2^n}$. Therefore $x = 2^nx^\prime$ is the desired $(C,\theta,n)$ vector.

If moreover $\{x_k\}_k$ is $(C^\prime,\{n_k\}_k)$ $\al$-RIS, obvious modifications yield that $x$ can be chosen to be a $(C^{\prime\prime},\theta,n)$ exact vector.

Arguing in the same way, for any $n\inn$, for any $F\subset\mathbb{N}$ with $\{\min\supp x_{n_k}: k\in F\}\in\mathcal{S}_n$ and $\{c_k\}_k\subset\mathbb{R}$, we have that
$\|\sum_{k\in F}c_kx_{n_k}\| > \frac{\theta}{2^n}\sum_{k\in F}|c_k|$.

The proof is exactly the same if $\bdx > 0$.

\end{proof}

\subsection*{Block sequences with $\al$-index zero}
In this subsection we show that sequences
$\{x_k\}_{k\inn}$ with $x_k$ a $(C,\theta,n_k)$ vector, with $\{n_k\}_k$ strictly increasing
have $\al$-index zero. Also also prove that sequences with $\al$-index zero have $\al$-RIS subsequences.

\begin{prp}
Let $\{x_k\}_k$ be a bounded block sequence in $\X$ with $\adx = 0$. Then it has a subsequence that is $(2C,\{n_k\}_k)$ $\al$-RIS, where $C = \sup\{\|x_k\|: k\inn\}$.\label{prpalzeroalris}
\end{prp}

\begin{proof}
Applying Proposition \ref{prop3.3} we have the following. There exists $j_0\inn$ such that for every $j\geqslant j_0$ there exists $k_j\inn$ such that for every $k\geqslant k_j$ and $\{\al_q\}_{q=1}^d$ very fast growing and $\mathcal{S}_j$ admissible sequence of $\al$-averages with $s(\al_q)>j_0$ for $q=1,\ldots,d$, we have that $\sum_{q=1}^d|\al_q(x_k)| < C$.

We shall show that for every $j\geqslant j_0$ and $k_0\inn$, there exists $k\geqslant k_0$ such that for every $f\in W$ of type I$_\al$ and $w(f) = n < j$, we have that $|f(x_k)| < \frac{2C}{2^n}$. If this is shown to be true, then by Remark \ref{rem3.11} we are done.

Fix $j\geqslant j_0$ and $k_0\inn$. Set $k = \max\{j_0, k_0, k_{j_0}\}$ and let $f\in W$ with $w(f) = n < j$. Then $f$ is of the form $f = \frac{1}{2^n}\sum_{q=1}^d\al_q$, where $\{\al_q\}_{q=1}^d$ is a very fast growing and $\mathcal{S}_n$ admissible sequence of $\al$-averages. We may clearly assume that $\ran \al_1\cap \ran x_k\neq\varnothing$. Then $\{\al_q\}_{q=2}^d$ is very fast growing with $s(\al_q) > \min\supp x_k \geqslant j_0$ and it is $S_j$ admissible, as it is $\mathcal{S}_n$ admissible and $n<j$. We conclude the following.

\begin{equation*}
|f(x_k)| \leqslant \frac{1}{2^n}(|\al_1(x_k)| + \sum_{q=2}^d|\al_q(x_k)|) <\frac{1}{2^n}(C + C)
\end{equation*}

\end{proof}

\begin{lem} Let $x = 2^n\sum_{k=1}^mc_kx_k$ be a $(C,\theta,n)$ vector in $\X$. Let also $\al$ be
an $\al$-average in $W$ and set $G_\al = \{k: \ran\al\cap\ran
x_k\neq\varnothing\}$. Then the following holds.
\begin{equation*}
|\al(x)| < \min\big\{\frac{C2^n}{s(\al)}\sum_{k\in G_\al}c_k,\;
\frac{6C}{s(\al)}\sum_{k\in G_\al}c_k + \frac{1}{3\cdot2^{2n}}\big\} + 2C2^n\max\{c_k:
k\in G_\al\}
\end{equation*}\label{lem3.6}
\end{lem}

\begin{proof}
If $\al = \frac{1}{p}\sum_{j=1}^df_j$. Set
\begin{eqnarray*}
E_1 &=& \{k\in G_\al:\;\text{there exists at most one}\;j\;\text{with}\;\ran f_j\cap\ran x_k\neq\varnothing\}\\
E_2 &=& \{1,\ldots,m\}\setminus E_1\\
J_k &=& \{j: \ran f_j\cap\ran
x_k\neq\varnothing\}\quad\text{for}\;k\in E_2.
\end{eqnarray*}
Then it is easy to see that
\begin{equation}
|\al(\sum_{k\in E_1}c_kx_k)| \leqslant \frac{C}{p}\sum_{k\in
G_\al}c_k \label{lem3.6eq1}
\end{equation}
Moreover
\begin{equation}
|\al(\sum_{k\in E_2}c_kx_k)| < 2C\max\{c_k: k\in
G_\al\}\label{lem3.6eq2}
\end{equation}
To see this, notice that
\begin{equation*}
|\al(\sum_{k\in E_2}c_kx_k)| \leqslant \frac{1}{p}\sum_{k\in
E_2}c_k\big(\sum_{j\in J_k}|f_j(x_k)|\big) < \max\{c_k: k\in
G_\al\}\frac{2Cp}{p}
\end{equation*}

Set $J = \{j:$ there exists $k\in E_1$ such that $\ran f_j\cap\ran
x_k\neq\varnothing\}$ and for $j\in J$ set $G_j = \{k\in E_1: \ran
f_j\cap\ran x_k\neq\varnothing\}$. Then the $G_j$ are pairwise
disjoint and $\cup_{j\in J}G_j = E_1$.

For $j\in J$, Corollary \ref{cor2.21} yields that
\begin{equation*}
|f_j(\sum_{k\in
G_j}c_kx_k)| \leqslant \frac{6C}{2^n}\sum_{k\in G_j}c_k + \frac{1}{3\cdot2^{3n}}
\end{equation*}
Therefore
\begin{equation}
|\al(\sum_{k\in E_1}c_kx_k)| \leqslant \frac{1}{p}\sum_{j\in
J}|f_j(\sum_{k\in G_j}c_kx_k)| \leqslant \frac{6C}{2^np}\sum_{k\in
G_\al}c_k + \frac{1}{3\cdot2^{3n}}\label{lem3.6eq3}
\end{equation}
Then \eqref{lem3.6eq1} and \eqref{lem3.6eq3} yield the following.
\begin{equation}
|\al(\sum_{k\in E_1}c_kx_k)| \leqslant
\min\big\{\frac{C}{s(\al)}\sum_{k\in G_\al}c_k,\;
\frac{6C}{2^ns(\al)}\sum_{k\in G_\al}c_k +
\frac{1}{3\cdot2^{3n}}\big\}\label{lem3.6eq4}
\end{equation}
By summing up \eqref{lem3.6eq2} and \eqref{lem3.6eq4} the result
follows.

\end{proof}

\begin{lem} Let $x$ be a $(C,\theta,n)$ vector in $\X$. Let also
$\{a_q\}_{q=1}^d$ be a very fast growing and
$\mathcal{S}_j$-admissible sequence of $\al$-averages, with $j<n$.
Then the following holds.
\begin{equation*}
\sum_{q=1}^d|\al_q(x)| < \frac{6C}{s(\al_1)} +
\frac{1}{2^n}
\end{equation*}\label{lem3.7}
\end{lem}

\begin{proof}
Assume that $x = 2^n\sum_{k=1}^mc_kx_k$ such that the assumptions of Definition \ref{defvector} are satisfied.
Set $q_1 = \min\{q: \ran \al_q\cap\ran x\neq\varnothing\}$. For
convenience assume that $q_1 = 1$. Then by Lemma \ref{lem3.6} we
have that
\begin{equation}
|\al_1(x)| < \frac{6C}{s(\al_1)} + \frac{7}{18\cdot 2^{2n}}
\label{lem3.7eq1}
\end{equation}
Set
\begin{eqnarray*}
J_1 &=& \{q>1:\;\text{there exists at most one}\;k\;\text{such that}\;\ran\al_q\cap\ran x_k\neq\varnothing\}\\
J_2 &=& \{q>1: q\notin J_1\}\\
G^q &=& \{k: \ran \al_q\cap\ran x_k\neq\varnothing\}\quad\text{for}\;q>1.\\
G_1 &=& \{k:\;\text{there exists}\;q\in J_1\;\text{with}\;\ran\al_q\cap\ran x_k\neq\varnothing\}
\end{eqnarray*}
Then $\{\min\supp x_k: k\in G_1\setminus\{\min
G_1\}\}\in\mathcal{S}_j$, hence $\sum_{k\in G_1}c_k< \frac{1}{18C2^{3n}}$.

It is easy to check that
\begin{equation}
\sum_{q\in J_1}|\al_q(x)| \leqslant 2^jC2^n\|\sum_{k\in
G_1}c_kx_k\| < 2^{n-1}C2^{n}\frac{1}{18C2^{3n}} = \frac{1}{36\cdot 2^n} \label{lem3.7eq2}
\end{equation}

For $q\in J_2$, Lemma \ref{lem3.6} yields that
\begin{eqnarray*}
|\al_q(x)| &<& \frac{C2^n}{s(\al_q)}\sum_{k\in G^q}c_k + 2C2^n\max\{c_k: k\in G^q\}\\
&<& \frac{C2^n}{\min\supp x}\sum_{k\in G^q}c_k + 2C2^nc_{k_q}
\end{eqnarray*}
where $k_q\in G^q$, such that $c_{k_q} = \max\{c_k: k\in G^q\}$.

Then $\{\min\supp x_{k_q}: q\in J_2\setminus\{\min
J_2\}\}\in\mathcal{S}_j$. By the above we conclude that
\begin{equation}
\sum_{q\in J_2}|\al_q(x)| < \frac{2C2^n}{\min\supp x} +
\frac{8}{36\cdot2^{2n}} < \frac{1}{4\cdot2^n} + \frac{8}{36\cdot2^{2n}} \label{lem3.7eq3}
\end{equation}
Summing up \eqref{lem3.7eq1}, \eqref{lem3.7eq2} and
\eqref{lem3.7eq3}, the desired result follows.

\end{proof}

\begin{prp}
Let $\{x_k\}_{k\inn}$ be a block sequence of $(C,\theta,n_k)$ vectors in $\X$ with $\{n_k\}_k$ strictly increasing. Then $\adx = 0$.\label{prop3.8}
\end{prp}

\begin{proof}
We shall make use of Proposition \ref{prop3.3}. Let $\e>0$ and
choose $j_0\inn$ such that $\frac{6C}{j_0} < \frac{\e}{2}$. For
$j\geqslant j_0$, choose $k_j$, such that $\frac{1}{2^{n_{k_j}}}<\frac{\e}{2}$. For $k\geqslant k_j$, Lemma
\ref{lem3.7} yields that if $\{\al_q\}_{q=1}^d$ is a very fast
growing and $\mathcal{S}_j$-admissible sequence of $\al$-averages
and $s(\al_q)>j_0$, for $q=1,\ldots,d$, we have that
\begin{equation*}
\sum_{q=1}^d|\al_q(x_k)| < \frac{6C}{j_0} + \frac{1}{2^{n_k}} <
\frac{\e}{2} + \frac{\e}{2} = \e
\end{equation*}
\end{proof}

\begin{prp} Let $x$ be a $(C,\theta,n)$ vector in $\X$. Then for any $f\in
W$ functional of type I$_\al$, such that $w(f) = j < n$, we have
that $|f(x)| < \frac{7C}{2^j}$.\label{cor3.9}
\end{prp}

\begin{proof}
Let $f = \frac{1}{2^j}\sum_{q=1}^d\al_j$ be a functional of type
I$_\al$ with weight $w(f) = j < n$. Then Lemma \ref{lem3.7} yields
that
\begin{equation*}
|f(x)| \leqslant \frac{1}{2^j}\big(
\sum_{q=1}^d|\al_q(x)|\big) <
\frac{1}{2^j}\big(\frac{6C}{s(\al_1)} + \frac{1}{2^n}\big)\leqslant \frac{7C}{2^j}
\end{equation*}

\end{proof}

The following proposition follows immediately from Proposition \ref{cor3.9} and Remark \ref{rem3.11}.

\begin{prp}
Let $\{x_k\}_{k\inn}$ be a block sequence of $(C,\theta,n_k)$ vectors in $\X$ with $\{n_k\}_k$ strictly increasing. Then passing, if necessary, to a subsequence, $\{x_k\}_{k\inn}$ is
$(7C,\{n_k\}_k)$\;$\al$-RIS. \label{rem3.12}
\end{prp}

\subsection*{Block sequences with $\be$-index zero.}
In this subsection we first prove that every block sequence of $(C,\theta,n_k)$ exact vectors with $\{n_k\}_k$ strictly increasing, has $\be$-index zero. This yields that every block
sequence has a further block sequence with both $\al, \be$ indices
equal to zero. We start with the following technical lemma. Its
meaning becomes more transparent in the following Corollary
\ref{corx3.14} and Lemmas \ref{lem3.14}, \ref{lem3.15}.

\begin{ntt} Let $x = 2^n\sum_{k=1}^mc_kx_k$ be a $(C,\theta,n)$ exact vector, with $\{x_k\}_{k=1}^m$ $(C,\{n_k\}_{k=1}^m)$ $\al$-RIS. Let also $f = \frac{1}{2}\sum_{j=1}^df_j$ be a type II functional. Set
\begin{eqnarray*}
I_0 &=& \{j: n\leqslant w(f_j) < 2^{2n}\}\\
I_1 &=& \{j: w(f_j) < n\}\\
I_2 &=& \{j: 2^{2n} \leqslant w(f_j) < n_1\}\\
J_k &=& \{j: n_k\leqslant w(f_j) <
n_{k+1}\},\;\text{for}\;k<m\;\text{and}\quad J_m = \{j:
n_m\leqslant w(f_j)\}
\end{eqnarray*}
\end{ntt}
Under the above notation the following lemma holds.

\begin{lem}
Let $x = 2^n\sum_{k=1}^mc_kx_k$ be a $(C,\theta,n)$ exact vector in $\X,
n\geqslant 2$, with $\{x_k\}_{k=1}^m$ $(C,\{n_k\}_{k=1}^m)$ $\al$-RIS. Let also $f = \frac{1}{2}\sum_{j=1}^df_j$ be a functional of type II.

Then there exists $F_f\subset \{k: \ran f\cap \ran
x_k\neq\varnothing\}$ with $\{\min\supp x_k: k\in
F_f\}\in\mathcal{S}_2$ such that
\begin{eqnarray*}
|f(x)| &<& 7C\#I_0 + \frac{C}{2}\bigg(\sum_{k=2}^m\sum_{j\in J_k}\frac{2^{n_k}}{2^{w(f_j) + n_{k-1}}} + \sum_{k=1}^{m-1}\sum_{j\in J_k}\frac{2^n}{2^{w(f_j)}}\\
 &&+ \sum_{j\in I_1}\frac{7}{2^{w(f_j)}}
 + \sum_{j\in I_2}\frac{2^n}{2^{w(f_j)}}\bigg) + C2^n\sum_{k\in F_f}c_k
\end{eqnarray*}\label{lem3.13}
\end{lem}

\begin{proof}
Notice that $\{J_k\}_{k=1}^m$ are disjoint intervals of
$\{1,\ldots,d\}$ and that $g_k = \frac{1}{2}\sum_{j\in J_k}f_j\in
W$, for $k=1,\ldots,m$.

Set $F_f = \{k: \ran g_k\cap\ran x_k\neq\varnothing\}$. It easily
follows that $\{\min\supp x_k: k\in F_f\}\in\mathcal{S}_2$ and
that
\begin{equation}
\frac{2^n}{2}\sum_{k=1}^m|\sum_{j\in J_k}f_j(c_kx_k)| \leqslant
C2^n\sum_{k\in F_f}c_k \label{lem3.13eq1}
\end{equation}
Let $k_0\leqslant m, j\in J_{k_0}$. Then
\begin{equation}
2^n|f_j(\sum_{k<k_0}c_kx_k)| < C\frac{2^{n_{k_0}}}{2^{w(f_j) +
n_{k_0 - 1}}}\;\;\text{and}\;\; 2^n|f_j(\sum_{k>k_0}c_kx_k)| <
C\frac{2^n}{2^{w(f_j)}} \label{lem3.13eq2}
\end{equation}

Proposition \ref{cor3.9} yields that for $j\in I_1$ we have that
$|f_j(x)| < \frac{7C}{2^{w(f_j)}}$ and hence
\begin{equation}
\frac{1}{2}\sum_{j\in I_1}|f_j(x)| < \frac{C}{2}\sum_{j\in
I_1}\frac{7}{2^{w(f_j)}} \label{lem3.13eq3}
\end{equation}

For $j\in I_2$ we have that $|f_j(x)| < \frac{C2^n}{2^{w(f_j)}}$ and therefore
\begin{equation}
\frac{1}{2}\sum_{j\in I_2}|f_j(x)| < \frac{C}{2}\sum_{j\in
I_2}\frac{2^n}{2^{w(f_j)}} \label{lem3.13eq4}
\end{equation}

By Remark \ref{remexactest} yields that $\|x\| < 7C$, and since
$I_0$ is an interval, it follows that $\frac{1}{2}\sum_{j\in
I_0}f_j\in W$. Therefore
\begin{equation}
\frac{1}{2}|\sum_{j\in I_0}f_j(x)| < 7C \label{lem3.13eq5}
\end{equation}

Summing up \eqref{lem3.13eq1} to \eqref{lem3.13eq5} the desired
result follows.

\end{proof}

The next corollary will be useful in the next sections, when we
define the notion of dependent sequences.

\begin{cor}
Let $x$ be a $(C,\theta,n)$ exact vector in $\X,
n\geqslant 3$. Let also $f = \frac{1}{2}\sum_{j=1}^df_j$ be a functional of type II. with $\widehat{w}(f)\cap\{n,\ldots,2^{2n}\} = \varnothing$.
Then
\begin{equation*}
|f(x)| < \frac{C}{2^n} + \frac{C}{2^{2n}} + \sum_{\{j:\;w(f_j)<n\}}\frac{4C}{2^{w(f_j)}}
\end{equation*}\label{corx3.14}
\end{cor}

\begin{proof}
Let $x = 2^n\sum_{k=1}^mc_kx_k$ with $\{x_k\}_{k=1}^m$ $(C,\{n_k\}_{k=1}^m)$ $\al$-RIS.
Apply Lemma \ref{lem3.13}. Then the following holds.
\begin{eqnarray}
|f(x)| &<& \frac{C}{2}\bigg(\sum_{k=2}^m\sum_{j\in J_k}\frac{2^{n_k}}{2^{w(f_j)
+ n_{k-1}}} + \sum_{k=1}^{m-1}\sum_{j\in J_k}\frac{2^n}{2^{w(f_j)}}\label{corx3.14eq1}\\
&&+ \sum_{j\in I_1}\frac{7}{2^{w(f_j)}} + \sum_{j\in
I_2}\frac{2^n}{2^{w(f_j)}}\bigg) + \frac{1}{36\cdot2^{2n}}\nonumber
\end{eqnarray}
Notice the following.
\begin{equation}
\sum_{k=2}^m\sum_{j\in J_k}\frac{2^{n_k}}{2^{w(f_j) + n_{k-1}}}
\leqslant \frac{1}{2^{n_1}} < \frac{1}{2^{2n}} \label{corx3.14eq2}
\end{equation}
\begin{equation}
\sum_{j\in I_2}\frac{2^n}{2^{w(f_j)}} + \sum_{k=1}^{m-1}\sum_{j\in
J_k}\frac{2^n}{2^{w(f_j)}} = 2^n\big(\sum_{\{j: w(f_j)\geqslant
2^{2n}\}}\frac{1}{2^{w(f_j)}}\big) \leqslant
\frac{2}{2^n}\label{corx3.14eq3}
\end{equation}
Applying \eqref{corx3.14eq2} and \eqref{corx3.14eq3} to
\eqref{corx3.14eq1} the result follows.
\end{proof}

\begin{lem}
Let $x = 2^n\sum_{k=1}^mc_kx_k$ be a $(C,\theta,n)$ exact vector in $\X,
n\geqslant 2$, with $\{x_k\}_{k=1}^m$ $(C,\{n_k\}_{k=1}^m)$ $\al$-RIS. Let also $\be$ be a $\be$-average. Then there exists $F_\be\subset
\{k: \ran \be\cap \ran x_k\neq\varnothing\}$ with $\{\min\supp
x_k: k\in F_\be\}\in\mathcal{S}_2$ such that
\begin{equation*}
|\be(x)| < \frac{8C}{s(\be)} + C2^n\sum_{k\in F_\be}c_k
\end{equation*}\label{lem3.14}
\end{lem}

\begin{proof}
If $\be = \frac{1}{p}\sum_{q=1}^dg_q$, then by definition the
$g_q$ are functionals of type II with disjoint weights
$\widehat{w}(g_q)$.

For convenience, we may write $g_q = \frac{1}{2}\sum_{j\in
G_q}f_j$, where the index sets $G_q, q=1,\ldots,d$ are pairwise
disjoint. Notice that for $j_1,j_2\in G, j_1\neq j_2$ we have that
$w(f_{j_1})\neq w(f_{j_2})$.

By slightly modifying the previously used notation, set $G =
\cup_{q=1}^dG_q$ and
\begin{eqnarray*}
I_0 &=& \{j\in G: n\leqslant w(f_j) < 2^{2n}\}\\
I_1 &=& \{j\in G: w(f_j) < n\}\\
I_2 &=& \{j\in G: 2^{2n} \leqslant w(f_j) < n_1\}\\
J_k &=& \{j\in G: n_k\leqslant w(f_j) <
n_{k+1}\},\;\text{for}\;k<m\quad\text{and}\\
J_m &=& \{j\in G: n_m\leqslant w(f_j)\}
\end{eqnarray*}
By Remark \ref{remark1.1} there exists at most one $q_0\leqslant
d$, with
$\widehat{w}(f_{q_0})\cap\{n,\ldots,2^{2n}\}\neq\varnothing$ and
if such a $q_0$ exists, then
$\#\widehat{w}(f_{q_0})\cap\{n,\ldots,2^{2n}\}\leqslant 1$.

Apply Lemma \ref{lem3.13}. Then for $q=1,\ldots,d$ there exists
$F_q \subset \{x_k: \ran \be\cap \ran x_k\neq\varnothing\}$ with
$\{\min\supp x_k: k\in F_q\}\in\mathcal{S}_2$ such that
\begin{eqnarray}
2^n|\be(x)| &<& \frac{7C}{p} + \frac{C}{2p}\bigg(\sum_{k=2}^m\sum_{j\in J_k}\frac{2^{n_k}}{2^{w(f_j) + n_{k-1}}}
+ \sum_{k=1}^{m-1}\sum_{j\in J_k}\frac{2^n}{2^{w(f_j)}}\label{lem3.14eq1}\\
&&+ \sum_{j\in I_1}\frac{7}{2^{w(f_j)}} + \sum_{j\in
I_2}\frac{2^n}{2^{w(f_j)}}\bigg) +
\frac{1}{p}C2^n\sum_{q=1}^d\sum_{k\in F_q}c_k\nonumber
\end{eqnarray}
Just as in the proof of Corollary \ref{corx3.14}, notice the
following.
\begin{equation}
\sum_{k=2}^m\sum_{j\in J_k}\frac{2^{n_k}}{2^{w(f_j) + n_{k-1}}}
 < \frac{1}{2^{2n}} \label{lem3.14eq2}
\end{equation}
\begin{equation}
\sum_{j\in I_2}\frac{2^n}{2^{w(f_j)}} + \sum_{k=1}^{m-1}\sum_{j\in
J_k}\frac{2^n}{2^{w(f_j)}} \leqslant
\frac{2}{2^n}\label{lem3.14eq3}
\end{equation}
By the definition of the coding function $\sigma$ we get
\begin{equation}
\sum_{j\in I_1}\frac{7}{2^{w(f_j)}} < \frac{7}{1000}
\label{lem3.14eq4}
\end{equation}
\begin{equation}
\frac{1}{p}C2^n\sum_{q=1}^d\sum_{k\in F_q}c_k\leqslant
C2^n\max\big\{\sum_{k\in F_q}c_k: q=1,\ldots,p\big\} =
C2^n\sum_{k\in F_{q_0}}c_k \label{lem3.14eq5}
\end{equation}
for some $1\leqslant q_0\leqslant d$.

Set $F_\be = F_{q_0}$ and apply \eqref{lem3.14eq2} to
\eqref{lem3.14eq5} to \eqref{lem3.14eq1} to derive the desired
result.

\end{proof}

\begin{lem}
Let $x$ be a $(C,\theta,n)$ exact vector in $\X,
n\geqslant 4$. Let also $\{\be_q\}_{q=1}^d$ be a very fast growing and
$\mathcal{S}_j$-admissible sequence of $\be$-averages with
$j\leqslant n-3$. Then we have that
\begin{equation*}
\sum_{q=1}^d|\be_q(x)| < \sum_{q=1}^d\frac{8C}{s(\be_q)} +
\frac{1}{2^n}
\end{equation*}\label{lem3.15}
\end{lem}

\begin{proof}
Set
\begin{eqnarray*}
J_1 &=& \{q:\;\text{there exists at most one}\; k\;\text{such that}\;\ran\be_q\cap\ran x_k\neq\varnothing\}\\
J_2 &=& \{1,\ldots,d\}\setminus J_2\\
G_1 &=& \{k:\;\text{there exists}\;q\in J_1\;\text{with}\;\ran\be_q\cap\ran x_k\neq\varnothing\}
\end{eqnarray*}

Then $\{\min\supp x_k: k\in G_1\setminus\{\min
G_1\}\}\in\mathcal{S}_{j+1}$ and it is easy to check that
\begin{equation}
\sum_{q\in J_1}^d|\be_q(x)| \leqslant 2^j2^n\|\sum_{k\in
G_1}c_kx_k\| < \frac{2}{9\cdot2^n} \label{lem3.15eq1}
\end{equation}
For $q\in J_2$, choose $F_q\subset\{1,\ldots,m\}$ as in Lemma
\ref{lem3.14} and set $F = \cup_{q\in J_2}F_q$. Then $\{\min\supp
x_k: k\in F\setminus\{\min F\}\}\in\mathcal{S}_{n-1}$, therefore
$\sum_{q\in J_2}\sum_{k\in F_q}c_k < \frac{1}{9C2^{3n}}$.

Lemma \ref{lem3.14} yields that
\begin{equation}
\sum_{q\in J_2}|\be_q(x)| < \sum_{q\in J_2}\frac{8C}{s(\be_q)}
+ \frac{1}{9\cdot2^{2n}} \label{lem3.15eq2}
\end{equation}
Combining \eqref{lem3.15eq1} and \eqref{lem3.15eq2}, the result
follows.

\end{proof}

\begin{prp} Let $\{x_k\}_k$ be a block sequence of $(C,\theta,n_k)$ exact vectors in $\X$ with $\{n_k\}_k$ strictly increasing. Then $\adx = 0$ as well as $\bdx = 0$.\label{cor3.16}
\end{prp}

\begin{proof}
Proposition \ref{prop3.8} yields that $\adx = 0$. To prove that
$\bdx = 0$, we shall make use of Proposition \ref{prop3.4}. Let
$\e>0$ and choose $j_0\inn$ such that
\begin{equation*}
\frac{8C}{j_0} < \frac{\e}{4}
\end{equation*}
For $j\geqslant j_0$ choose $k_j$, such that $n_{k_j} \geqslant
j + 3$ and $\frac{1}{2^{n_{k_j}}}<\frac{\e}{4}$. For
$k\geqslant k_j$, Lemma \ref{lem3.15} yields that if
$\{\be_q\}_{q=1}^d$ is a very fast growing and
$\mathcal{S}_j$-admissible sequence of $\be$-averages and
$s(\be_q)>j_0$, for $q=1,\ldots,d$, we have that
\begin{equation*}
\sum_{q=1}^d|\be_q(x_k)| < \sum_{q=1}\frac{8C}{s(\be_q)} + \frac{1}{2^{n_k}} < \frac{\e}{4} + \sum_{j>\min\supp x_k}\frac{8C}{2^j} + \frac{\e}{4} < \e
\end{equation*}
\end{proof}

\begin{cor} Let $\{x_k\}_{k\inn}$ be a normalized block sequence in $\X$. Then
there exists a further normalized block sequence $\{y_k\}_{k\inn}$
of $\{x_k\}_{k\inn}$, such that $\ady = 0$ as well as $\bdy =
0$.\label{cor3.17}
\end{cor}

\begin{proof}
If $\adx = 0$ and $\bdx = 0$, then there is nothing to prove.
Otherwise, if $\adx > 0$ or $\bdx > 0$, apply Proposition
\ref{prop3.5} to construct a block sequence
$\{z_k\}_{k\inn}$ of $(1,\theta,n_k)$ vectors, with $\{n_k\}_k$ strictly increasing.  Then by Proposition \ref{prop3.8} $\adz = 0$ and Proposition \ref{rem3.12} yields,
that passing, if necessary, to a subsequence, we have that
$\{z_k\}_{k\inn}$ is $(7,\{n_k\}_k)$\;$\al$-RIS.

If $\bdz = 0$, set $y_k = \frac{1}{\|z_k\|}z_k$ and
$\{y_k\}_{k\inn}$ is the desired sequence.

Otherwise, if $\bdz > 0$, apply once more Proposition
\ref{prop3.5} to construct a block sequence
$\{w_k\}_{k\inn}$, of $(7,\theta^\prime,m_k)$ exact vectors, with $\{m_k\}_k$ strictly increasing. Proposition \ref{cor3.16} yields that $\adw = 0$, as well as $\bdw = 0$. Set $y_k = \frac{1}{\|w_k\|}w_k$ and
$\{y_k\}_{k\inn}$ is the desired sequence.

\end{proof}

\section{$c_0$ spreading models}

This section is devoted to necessary conditions for a sequence
$\{x_k\}_k$ to generate a $c_0$ spreading model. At the beginning
a Ramsey type result is proved concerning type II functionals
acting on a block sequence $\{x_k\}_k$ with $\bdx = 0$. Then
conditions are provided for a finite sequence to be equivalent to
the basis of $\ell_\infty^n$. This is critical for establishing
the HI property and the properties of the operators in the space.
Moreover it is shown that any block sequence $\{x_k\}_k$ with
$\adx = 0$ and $\bdx = 0$ contains a subsequence generating a
$c_0$ spreading model. Another critical property related to
sequences generating $c_0$ spreading models is that increasing
Schreier sums of them define $\al$-RIS sequences.

\subsection*{Evaluation of type II functionals on $\{x_k\}_k$ with $\bdx = 0$}

\begin{dfn}
Let $x_1<x_2<x_3$ be vectors in $\X$, $f =
\frac{1}{2}\sum_{j=1}^df_j$ be a functional of type II, such that
$\supp f\cap\ran x_i\neq\varnothing$, for $i=1,2,3$ and $j_0 =
\min\{j: \ran f_j\cap \ran x_2\neq\varnothing\}$. If $\ran
f_{j_0}\cap\ran x_3 = \varnothing$, then we say that $f$ separates
$x_1,x_2,x_3$.
\end{dfn}

\begin{dfn}
Let $i,j\inn$. If there exists $f\in W$ a functional of type II,
such that $i,j\in\widehat{w}(f)$, then we say that $i$ is
compatible to $j$.
\end{dfn}

\begin{lem}
Let $x_1 < x_2 <\cdots<x_m$ be vectors in $\X$, such that there
exist $\e>0$ and $\{f_k\}_{k=2}^{m-1}$ functionals of type II
satisfying the following.
\begin{enumerate}

\item[(i)] $f_k$ separates $x_1,x_k,x_m$, for $k=2,\ldots,m-1$

\item[(ii)] If $f_k = \frac{1}{2}\sum_{j=1}^{d_k}f_j^k$ and $j_k =
\min\{j: \ran f_j^k\cap\ran x_k\neq\varnothing\}$, then
$w(f_{j_k}^k)$ is not compatible to $w(f_{j_\ell}^\ell)$ for
$k\neq\ell$.

\item[(iii)] $|f_k(x_m)| > \e$ for $k=2,\ldots,m-1$

\end{enumerate}

Then there exists a $\be$-average $\be$ in $W$ of size $s(\be) =
m-2$ such that $\be(x_m) > \e$.\label{lem4.3}
\end{lem}

\begin{proof}
Set $g_k = \sgn(f_k(x_m))f_k|_{\ran x_m}$, for $k=2,\ldots,m-1$.
Then $g_k$ is a functional of type II in $W$. We will show that
the $g_k$ have disjoint weights $\widehat{w}(g_k)$.

Towards a contradiction, suppose that there exist $2\leqslant k <
\ell\leqslant m-1$ and
$i\in\widehat{w}(g_k)\cap\widehat{w}(g_\ell)$. By (i) and the way
type II functionals are constructed, it follows that
$f_k|_{[\min\supp x_2,\ldots,\max\supp f^\ell_{j_\ell}]} =
\substack{+\\[-2pt]-}f_\ell|_{[\min\supp x_2,\ldots,\max\supp f^\ell_{j_\ell}]}$. This
contradicts (ii).

By the above, it follows that if we set $\be =
\frac{1}{m-2}\sum_{k=2}^{m-1}g_k$, then $\be$ is the desired
$\be$-average.

\end{proof}

\begin{lem}
Let $x_1 < x_2 <\cdots<x_m$ be vectors in $\X$, such that there
exist $\e>0$ and $\{f_k\}_{k=2}^{m-1}$ functionals of type II
satisfying the following.
\begin{enumerate}

\item[(i)] $f_k$ separates $x_1,x_k,x_m$, for $k=2,\ldots,m-1$

\item[(ii)] If $f_k = \frac{1}{2}\sum_{j=1}^{d_k}f_j^k$ and $j_k =
\min\{j: \ran f_j^k\cap\ran x_k\neq\varnothing\}$, then
$w(f_{j_k}^k) = w(f_{j_\ell}^\ell)$ for $k\neq\ell$.

\item[(iii)] If $j_k^\prime = \min\{j: \ran f_j^k\cap\ran
x_m\neq\varnothing\}$, then $w(f_{j_k^\prime}^k) \neq
w(f_{j_\ell^\prime}^\ell)$ for $k\neq\ell$.

\item[(iv)] $|f_k(x_m)| > \e$ for $k=2,\ldots,m-1$

\end{enumerate}

Then there exists a $\be$-average $\be$ in $W$ of size $s(\be) =
m-2$ such that $\be(x_m) > \e$.\label{lem4.4}
\end{lem}

\begin{proof}
As before, set $g_k = \sgn(f_k(x_m))f_k|_{\ran x_m}$, for
$k=2,\ldots,m-1$. Then $g_k$ is a functional of type II in $W$. We
will show that the $g_k$ have disjoint weights $\widehat{w}(g_k)$.

Suppose that there exist $2\leqslant k < \ell\leqslant m-1$ and
$i\in\widehat{w}(g_k)\cap\widehat{w}(g_\ell)$. By (i), (ii) and
the way type II functionals are constructed, it follows that
\begin{equation*}
f_k|_{[\min\supp x_2,\ldots,\min\supp x_{m}]} = \substack{+\\[-2pt]-}f_\ell|_{[\min\supp x_2,\ldots,\min\supp x_{m}]}
\end{equation*}
This leaves us no choice, but to conclude that
$w(f_{j_k^\prime}^k) = w(f_{j_\ell^\prime}^\ell)$, a
contradiction.

It follows that if we set $\be =
\frac{1}{m-2}\sum_{k=2}^{m-1}g_k$, then $\be$ is the desired
$\be$-average.

\end{proof}

\begin{prp}
Let $\{x_k\}_{k\inn}$ be a bounded block sequence in $\X$, such
that $\bdx = 0$. Then for any $\e>0$, there exists $M$ an infinite
subset of the naturals, such that for any $k_1 < k_2 < k_3\in M$,
for any functional $f\in W$ of type II that separates
$x_{k_1},x_{k_2},x_{k_3}$, we have that $|f(x_{k_i})| < \e$, for
some $i\in\{1,2,3\}$.\label{prop4.5}
\end{prp}

\begin{proof}
Suppose that this is not the case. Then by using Ramsey theorem
\cite{Ra}, we may assume that there exists $\e>0$ such that for
any $k < \ell < m\inn$, we have that there exists $f_{k,\ell,m}$ a
functional of type II, that separates $x_k, x_\ell, x_m$ and
$|f_{k,\ell,m}(x_k)| > \e, |f_{k,\ell,m}(x_\ell)| > \e$ and
$|f_{k,\ell,m}(x_m)| > \e$.

For $1<k<m$, if $f_{1,k,m} = \frac{1}{2}\sum_{j=1}^{d_{k,m}}f_j^{k,m}$, set
\begin{eqnarray*}
i_{k,m} &=& \min\{j: \ran f_j^{k,m}\cap\ran x_1\neq\varnothing\}\\
j_{k,m} &=& \min\{j: \ran f_j^{k,m}\cap\ran x_k\neq\varnothing\}\\
j_{k,m}^\prime &=& \min\{j: \ran f_j^{k,m}\cap\ran x_m\neq\varnothing\}
\end{eqnarray*}

Notice, that for $1<k<m$, since $|f_{1,k,m}(x_1)| > \e$, it follows that
\begin{equation*}
\frac{1}{2^{w(f^{k,m}_{i_{k,m}})}} > \frac{\e}{\|x_1\|\max\supp x_1}
\end{equation*}

By applying Ramsey theorem once more, we may assume that there
exists $n_1\inn$, such that for any $1<k<m$, we have that
$w(f^{k,m}_{i_{k,m}}) = n_1$

Arguing in the same way and diagonalizing, we may assume that for
any $k>1$, there exists $n_k\inn$ such that for any $m>k$, we have
that $w(f^{k,m}_{j_{k,m}}) = n_k$. Set
\begin{eqnarray*}
A_1 &=& \big\{\{k,\ell\}\in[\mathbb{N}\setminus\{1\}]^2: n_k\neq n_\ell\;\text{and}\;n_k\;\text{is compatible to}\;n_\ell\big\}\\
A_2 &=& \big\{\{k,\ell\}\in[\mathbb{N}\setminus\{1\}]^2: n_k\neq n_\ell\;\text{and}\;n_k\;\text{is not compatible to}\;n_\ell\big\}\\
A_3 &=& \big\{\{k,\ell\}\in[\mathbb{N}\setminus\{1\}]^2: n_k = n_\ell\big\}
\end{eqnarray*}
Once more, Ramsey theorem yields that there exists $M$ an infinite
subset of the naturals, such that $[M]^2\subset A_1, [M]^2\subset
A_2$, or $[M]^2\subset A_3$.

Assume that $[M]^2\subset A_1$ and for convenience assume that $M
= \mathbb{N}\setminus\{1\}$. Choose $k_0>1$ such that $k_0 >
\max\supp x_1$. Since $n_1$ is compatible to $n_2$ and in general
$n_{k-1}$ is compatible to $n_{k}$, for $k>1$, it follows that
there exists a functional $f = \frac{1}{2}\sum_{j=1}^df_j$ of type
II in $W$, such that $\ran f\cap \ran x_1\neq\varnothing$ and for
$k=1,\ldots,k_0$ there exists $j_k$, with $w(f_{j_k}) = n_k$, for
$k=1,\ldots,k_0$.

Since $\min\supp f_1\leqslant\max\supp x_1$ it follows that
$\{f_j\}_{j=1}^d$ can not be $\mathcal{S}_1$-admissible, a
contradiction.

Assume next that $[M]^2\subset A_2$. Lemma \ref{lem4.3} yields
that $\bdx > 0$ and since this cannot be, we conclude that
$[M]^2\subset A_3$, therefore there exists $n_0\inn$, such that
$n_k = n_0$, for any $k\in M$.

Assume once more that $M = \mathbb{N}\setminus\{1\}$ and set
\begin{equation*}
B = \big\{\{k,\ell,m\}\in[\mathbb{N}\setminus\{1\}]^3: w(f^{1,k,m}_{j_{k,m}^\prime}) = w(f^{1,\ell,m}_{j_{\ell,m}^\prime})\big\}
\end{equation*}

If there exists $M$ an infinite subset of the naturals, such that
$[M]^3\subset B^c$, Lemma \ref{lem4.4} yields that $\bdx > 0$,
therefore by one last Ramsey argument, there exists $M$ an
infinite subset of the naturals, such that $[M]^3\subset B$.

By the above, we conclude that for $m\geqslant4$, $\ran x_k\subset
\ran f^{2,m}_{j_{2,m}}$ and $|f^{2,m}_{j_{2,m}}(x_k)| > 2\e$, for
$k= 2,\ldots,m-2$.

Set $f_m = f^{2,m}_{j_{2,m}}$ and let $f$ be a $w^*$ limit of some
subsequence of $\{f_m\}_{m\inn}$. Then $|f(x_k)|\geqslant 2\e$,
for any $k\geqslant 2$. Corollary \ref{cor2.22} yields a
contradiction and this completes the proof.

\end{proof}

\begin{rmk}
The proof of Proposition \ref{prop4.5} is the only place where the
condition $\bdx = 0$ is needed. This makes necessary to introduce
the $\be$-averages and their use in the definition of the norm.
\end{rmk}

\subsection*{Finite sequences equivalent to $
\ell_\infty^n$ basis.}

\begin{prp} Let $x_1<\cdots<x_n$ be a seminormalized block sequence in $\X$,
such that $\|x_k\|\leqslant 1$ for $k=1,\ldots,n$ and there exist
$n+3\leqslant j_1 <\cdots< j_n$ strictly increasing naturals such
that the following are satisfied.
\begin{enumerate}

\item[(i)] For any $k_0\in\{1,\ldots,n\}$, for any $k\geqslant
k_0, k\in\{1,\ldots,n\}$, for any $\{\al_q\}_{q=1}^d$ very fast
growing and $\mathcal{S}_j$-admissible sequence of $\al$-averages,
with $j<j_{k_0}$ and $s(\al_1) > \min\supp x_{k_0}$, we have that
$\sum_{q=1}^d|\al_q(x_k)| < \frac{1}{n\cdot2^n}$.

\item[(ii)] For any $k_0\in\{1,\ldots,n\}$, for any $k\geqslant
k_0, k\in\{1,\ldots,n\}$, for any $\{\be_q\}_{q=1}^d$ very fast
growing and $\mathcal{S}_j$-admissible sequence of $\be$-averages,
with $j<j_{k_0}$ and $s(\be_1) > \min\supp x_{k_0}$, we have that
$\sum_{q=1}^d|\be_q(x_k)| < \frac{1}{n\cdot2^n}$.

\item[(iii)] For $k = 1,\ldots,n-1$, the following holds:
$\frac{1}{2^{j_{k+1}}}\max\supp x_k < \frac{1}{2^n}$.

\item[(iv)] For any $1\leqslant k_1 < k_2 < k_3\leqslant n$, for
any functional $f\in W$ of type II that separates
$x_{k_1},x_{k_2},x_{k_3}$, we have that $|f(x_{k_i})| <
\frac{1}{n\cdot2^n}$, for some $i\in\{1,2,3\}$.

\end{enumerate}

Then $\{x_k\}_{k=1}^n$ is equivalent to $\ell_\infty^n$ basis,
with an upper constant $3 + \frac{3}{2^n}$. Moreover, for any
functional $f\in W$ of type I$_\al$ with weight $w(f) = j < j_1$,
we have that $|f(\sum_{k=1}^nx_k)| < \frac{3 +
\frac{4}{2^n}}{2^j}$.\label{prop4.6}
\end{prp}

\begin{proof}
By using Remark \ref{remark1.2}, we will inductively prove, that
for any $\{c_k\}_{k=1}^n\subset[-1,1]$ the following hold.
\begin{enumerate}

\item[(i)] For any $f\in W$, we have that $|f(\sum_{k=1}^nc_kx_k)|
< (3 + \frac{3}{2^n})\max\{|c_k|:k=1,\ldots,n\}$.

\item[(ii)] If $f$ is of type I$_\al$ and $w(f)\geqslant 2$, then
$|f(\sum_{k=1}^nc_kx_k)| < (1 +
\frac{2}{2^n})\max\{|c_k|:k=1,\ldots,n\}$.

\item[(iii)] If $f$ is of type I$_\al$ and $w(f) = j < j_1$, then
$|f(\sum_{k=1}^nc_kx_k)| < \frac{3 +
\frac{4}{2^n}}{2^j}\max\{|c_k|:k=1,\ldots,n\}$.

\end{enumerate}

For any functional $f\in W_0$ the inductive assumption holds.
Assume that it holds for any $f\in W_m$ and let $f\in W_{m+1}$. If
$f$ is a convex combination, then there is nothing to prove.

Assume that $f$ is of type I$_\al, f =
\frac{1}{2^j}\sum_{q=1}^d\al_q$, where $\{\al_q\}_{q=1}^d$ is a
very fast growing and $\mathcal{S}_j$-admissible sequence of
$\al$-averages in $W_m$.

Set $k_1 = \min\{k: \ran f\cap\ran x_k\neq\varnothing\}$ and $q_1
= \min\{q: \ran \al_q\cap\ran x_{k_1}\neq\varnothing\}$.

We distinguish 3 cases.\vskip3pt

\noindent {\em Case 1:} $j<j_1$.

For $q>q_1$, we have that $s(\al_q) > \min\supp x_{k_1}$,
therefore we conclude that
\begin{equation}
\sum_{q>q_1}|\al_q(\sum_{k=1}^nc_kx_k)| <
\frac{1}{2^n}\max\{|c_k|:k=1,\ldots,n\} \label{prop4.6eq1}
\end{equation}
while the inductive assumption yields that
\begin{equation}
|\al_{q_1}(\sum_{k=1}^nc_kx_k)| < (3 +
\frac{3}{2^n})\max\{|c_k|:k=1,\ldots,n\} \label{prop4.6eq2}
\end{equation}
Then \eqref{prop4.6eq1} and \eqref{prop4.6eq2} allow us to
conclude that
\begin{equation}
|f(\sum_{k=1}^nc_kx_k)| < \frac{3 +
\frac{4}{2^n}}{2^j}\max\{|c_k|:k=1,\ldots,n\}\label{prop4.6eqx1}
\end{equation}
Hence, (iii) from the inductive assumption is satisfied.\vskip3pt

\noindent {\em Case 2:} There exists $k_0 < n$, such that
$j_{k_0}\leqslant j < j_{k_0 + 1}$.

Arguing as previously we get that
\begin{equation}
|f(\sum_{k>k_0}c_kx_k)| < \frac{3 +
\frac{4}{2^n}}{2^{j_{k_0}}}\max\{|c_k|:k=1,\ldots,n\}
\label{prop4.6eq3}
\end{equation}
and
\begin{equation}
|f(\sum_{k<k_0}c_kx_k)| < \frac{1}{2^n}\max\{|c_k|:k=1,\ldots,n\}
\label{prop4.6eq4}
\end{equation}
Using \eqref{prop4.6eq3}, \eqref{prop4.6eq4}, the fact that
$|f(x_{k_0})| \leqslant 1$ and $j_{k_0}\geqslant n+3$, we conclude
that
\begin{equation}
|f(\sum_{k=1}^nc_kx_k)| < (1 +
\frac{2}{2^n})\max\{|c_k|:k=1,\ldots,n\}\label{prop4.6eqx2}
\end{equation}

\noindent {\em Case 3:} $j\geqslant j_n$

By using the same arguments, we conclude that
\begin{equation}
|f(\sum_{k=1}^nc_kx_k)| < (1 +
\frac{1}{2^n})\max\{|c_k|:k=1,\ldots,n\}\label{prop4.6eqx3}
\end{equation}

Then \eqref{prop4.6eqx1}, \eqref{prop4.6eqx2} and
\eqref{prop4.6eqx3} yield that (ii) from the inductive assumption
is satisfied.

If $f$ is of type I$_\be$, then the proof is exactly the same,
therefore assume that $f$ is of type II, $f =
\frac{1}{2}\sum_{j=1}^df_j$, where $\{f_j\}_{j=1}^d$ is an
$\mathcal{S}_1$-admissible sequence of functionals of type I$_\al$
in $W_m$. Set
\begin{eqnarray*}
E &=& \{k: |f(x_k)| \geqslant \frac{1}{n\cdot2^n}\}\\
E_1 &=& \{k\in E:\;\text{there exist at least two}\;j\;\text{such
that}\;\ran f_j\cap\ran x_k\neq\varnothing\}
\end{eqnarray*}
Then $\#E_1\leqslant 2$. Indeed, if $k_1<k_2<k_3\in E_1$, then $f$
separates $x_{k_1},x_{k_2}$ and $x_{k_3}$ which contradicts our
initial assumptions.

If moreover we set $J = \{j:$ there exists $k\in E\setminus E_1$
such that $\ran f_j\cap\ran x_k\neq\varnothing\}$, then for the
same reasons we get that $\#J\leqslant 2$.

Since for any $j$, we have that $w(f_j)\in L$, we get that
$w(f_j)>2$, therefore:
\begin{eqnarray}
|f(\sum_{k \in E\setminus E_1}^nc_kx_k)| &<& (1 + \frac{2}{2^n})\max\{|c_k|:k=1,\ldots,n\}\label{prop4.6eq5}\\
|f(\sum_{k \in  E_1}^nc_kx_k)| &\leqslant& 2 \max\{|c_k|:k=1,\ldots,n\}\label{prop4.6eq6}\\
|f(\sum_{k \notin  E}^nc_kx_k)| &\leqslant&
n\cdot\frac{1}{n\cdot2^n}\max\{|c_k|:k=1,\ldots,n\}\label{prop4.6eq7}
\end{eqnarray}
Finally, \eqref{prop4.6eq5} to \eqref{prop4.6eq7} yield the
following.
\begin{equation*}
|f(\sum_{k=1}^nc_kx_k)| < (3 + \frac{3}{2^n})\max\{|c_k|:k=1,\ldots,n\}
\end{equation*}
This means that (i) from the inductive assumption is satisfied an
this completes the proof.

\end{proof}

\subsection*{The spreading models of $\X$.}
In this subsection we show that every seminormalized block
sequence has a subsequence which generates either $\ell_1$ or
$c_0$ as a spreading model.

\begin{prp}
Let $\{x_k\}_{k\inn}$ be a seminormalized block sequence in $\X$,
such that $\|x_k\|\leqslant 1$ for all $k\inn$ and $\adx = 0$ as
well as $\bdx = 0$. Then it has a subsequence, again denoted by
$\{x_k\}_{k\inn}$ satisfying the following.
\begin{enumerate}

\item[(i)] $\{x_k\}_{k\inn}$ generates a $c_0$ spreading model.
More precisely, for any $n\leqslant k_1<\cdots<k_n$, we have that
$\|\sum_{i=1}^nx_{k_i}\| \leqslant 4$.

\item[(i)] There exists a strictly increasing sequence of naturals
$\{j_n\}_{n\inn}$, such that for any $n\leqslant k_1 < \cdots
<k_n$, for any functional $f$ of type I$_\al$ with $w(f) = j<j_n$,
we have that
    \begin{equation*}
    |f(\sum_{i=1}^nx_{k_i})| < \frac{4}{2^j}
    \end{equation*}
\end{enumerate}\label{cor4.7}
\end{prp}

\begin{proof}
By repeatedly applying Proposition \ref{prop4.5} and
diagonalizing, we may assume that for any $n\leqslant k_1 < k_2 <
k_3$, for any functional $f$ of type II that separates
$x_{k_1},x_{k_2}$ and $x_{k_3}$, we have that $|f(x_{k_i})| <
\frac{1}{n\cdot2^n}$, for some $i\in\{1,2,3\}$.

Use Propositions \ref{prop3.3} and \ref{prop3.4} to inductively
choose a subsequence of $\{x_k\}_{k\inn}$, again denoted by
$\{x_k\}_{k\inn}$ and $\{j_k\}_{k\inn}$ a strictly increasing
sequence of naturals with $j_k\geqslant k+3$ for all $k\inn$, such
that the following are satisfied.
\begin{enumerate}

\item[(i)] For any $k_0\inn$, for any $k\geqslant k_0$, for any
$\{\al_q\}_{q=1}^d$ very fast growing and
$\mathcal{S}_j$-admissible sequence of $\al$-averages, with
$j<j_{k_0}$ and $s(\al_1) > \min\supp x_{k_0}$, we have that
$\sum_{q=1}^d|\al_q(x_k)| < \frac{1}{k_0\cdot2^{k_0}}$.

\item[(ii)] For any $k_0\inn$, for any $k\geqslant k_0$, for any
$\{\be_q\}_{q=1}^d$ very fast growing and
$\mathcal{S}_j$-admissible sequence of $\be$-averages, with
$j<j_{k_0}$ and $s(\be_1) > \min\supp x_{k_0}$, we have that
$\sum_{q=1}^d|\be_q(x_k)| < \frac{1}{k_0\cdot2^{k_0}}$.

\item[(iii)] For $k\inn$, the following holds:
$\frac{1}{2^{j_{k+1}}}\max\supp x_k < \frac{1}{2^k}$.
\end{enumerate}
It is easy to check that for $n\leqslant k_1<\cdots<k_n$, the
assumptions of Proposition \ref{prop4.6} are satisfied.

\end{proof}

Propositions \ref{prop3.5} and \ref{cor4.7} yield the following.

\begin{cor}
Let $\{x_k\}_{k\inn}$ be a normalized weakly null sequence in
$\X$. Then it has a subsequence that generates a spreading model
which is either equivalent to $c_0$, or to $\ell_1$.\label{cor4.8}
\end{cor}

\begin{dfn}
A pair $\{x,f\}$, where $x\in\X$ and $f\in W$ is called an $(n,1)$-exact pair, if the following hold.

\begin{itemize}

\item[(i)] $f$ is a functional of type I$_\al$ with $w(f) = n, \min\supp x\leqslant\min\supp f$ and $\max\supp x\leqslant \max\supp f$.

\item[(ii)] There exists $x^\prime\in\X$ a $(4,1,n)$ exact vector, such that $\frac{28}{29}<f(x^\prime)\leqslant 1$ and $x = \frac{x^\prime}{f(x^\prime)}$.

\end{itemize}

A pair $\{x,f\}$, where $x\in\X$ and $f\in W$ is called an $(n,0)$-exact pair, if the following hold.

\begin{itemize}

\item[(i)] $f$ is a functional of type I$_\al$ with $w(f) = n, \min\supp x\leqslant\min\supp f$ and $\max\supp x\leqslant \max\supp f$.

\item[(ii)] $x$ is a $(4,1,n)$ exact vector and $f(x) = 0$.

\end{itemize}
\label{defexactpair}
\end{dfn}

\begin{rmk}
If $\{x,f\}$ is an $(n,1)$-exact pair, then $f(x) = 1$ and by Remark \ref{remexactest} we have that $1\leqslant\|x\|\leqslant 29$.\label{remestexactpair}
\end{rmk}

\begin{prp}
Let $\{x_k\}_{k\inn}$ be a normalized block sequence in $\X$, that
generates a $c_0$ spreading model. Then there exists
$\{F_k\}_{k\inn}$ an increasing sequence of subsets of the
naturals such that $\#F_k\leqslant\min F_k$ for all $k\inn$ and
$\lim_k\#F_k = \infty$ such that by setting $y_k = \sum_{i\in
F_k}x_k$, there exists a subsequence of $\{y_k\}_{k\inn}$, which
generates an $\ell_1^n$ spreading model, for all $n\inn$.

Furthermore, for any $k_0,n\inn$, there exists $F$ a
finite subset of $\mathbb{N}$ with $\min F\geqslant k_0$ and
$\{c_k\}_{k\in F}$, such that
\begin{enumerate}

\item[(i)]$x^\prime = 2^n\sum_{k\in F}c_ky_k$ is a $(4,1,n)$ exact vector.

\item[(ii)] For any $\eta > 0$ there exists a functional $f_\eta$ of
type I$_\al$ of weight $w(f_\eta) = n$ such that $f_\eta(x^\prime) > 1-\eta, \min\supp x^\prime\leqslant\min\supp f_\eta$ and $\max\supp f_\eta > \max\supp x^\prime$.

\end{enumerate}
In particular, if $f = f_{1/29}, x = \frac{x^\prime}{f(x^\prime)}$, then $\{x,f\}$ is an $(n,1)$-exact pair.
\label{prop4.9}
\end{prp}

\begin{proof}

Since $\{x_k\}_{k\inn}$ generates a $c_0$ spreading model,
Proposition \ref{prop3.5} yields that $\adx = 0$ as well as $\bdx
= 0$, therefore passing, if necessary, to a subsequence
$\{x_k\}_{k\inn}$, satisfies the conclusion of Proposition
\ref{cor4.7}.

Choose $\{F_k\}_{k\inn}$ an increasing sequence of subsets of the
naturals, such that the following are satisfied.
\begin{enumerate}

\item[(i)] $\#F_k \leqslant \min F_k$ for all $k\inn$.

\item[(ii)] $\#F_{k+1} > \max\big\{\#F_k,\;2^{\max\supp x_{\max
F_k}}\big\}$, for all $k\inn$.

\end{enumerate}

By Proposition \ref{prop4.5} and Remark \ref{rem3.11}, we have
that $1 \leqslant \|y_k\| \leqslant 4$, for all $k\inn$ and
passing, if necessary, to a subsequence, $\{y_k\}_{k\inn}$ is
$(4,\{n_k\}_{k\inn})$\;$\al$-RIS.

Moreover it is easy to see, that for any $k\inn, \eta > 0$, there
exists an $\al$-average $\al$ of size $s(\al) = \#F_k$, such that
$\al(y_k) > 1 - \eta$ and $\ran\al\subset y_k$.

This yields that $\ady > 0$, therefore we may apply Proposition
\ref{prop3.5} to conclude that $\{y_k\}_{k\inn}$ has a subsequence
generating an $\ell_1^n$ spreading model, for all $n\inn$.

We now prove the second assertion. Let $k_0,n\inn$ and fix $0< \e < (36\cdot4\cdot2^{3n})^{-1}$.
By taking a larger $k_0$, we may assume that $n_{k_0} > 2^{2n}$.

Set $\e^\prime = \e(1-\e)$ Proposition \ref{prop1.5} yields that
there exists $\{d_1,\ldots,d_m\}$ a finite subset of $\{k:
k\geqslant k_0\}$ and $\{c_k\}_{k=1}^m$ such that $x^{\prime\prime} =
\sum_{k=1}^mc_ky_{d_k}$ is a $(n,\e^\prime)$ s.c.c. It is
straightforward to check that $x^\prime =
2^n\sum_{k=1}^{m-1}\frac{c_k}{1-c_m}y_{d_k}$ is a $(4,\theta,n)$ exact vector. If (ii) also holds, it will follow that $\theta \geqslant 1$.

For some $\eta>0$, $k = 1,\ldots,m$, choose an $\al$-average
$\al_k$ of size $s(\al_k) = \#F_{d_k}$, such that $\al_k(y_{d_k})>
1 - \eta$ and $\ran\al\subset y_k$. Set $f = \frac{1}{2^n}(\sum_{k
= 1}^m\al_k)$, which is a functional of type I$_\al$ of weight
$w(f) =n$ such that $f(x^\prime) > 1 - \eta$ and $\max\supp f
> \max\supp x$.

\end{proof}

\begin{cor}
Let $Y$ be an infinite dimensional closed subspace of $\X$. Then
$Y$ admits a spreading model equivalent to $c_0$ as well as a
spreading model equivalent to $\ell_1$.\label{cor4.10}
\end{cor}

\begin{proof}
Assume first that $Y$ is generated by some normalized block
sequence $\{x_k\}_{k\inn}$. Corollary \ref{cor3.17} and
Proposition \ref{cor4.7} yield that it has a further normalized
block sequence $\{y_k\}_{k\inn}$, generating a spreading model
equivalent to $c_0$.

Proposition \ref{prop4.9} yields that $\{y_k\}_{k\inn}$ has a
further block sequence generating an $\ell_1$ spreading model.

Since any subspace contains a sequence arbitrarily close to a
block sequence, the result follows.

\end{proof}

Corollary \ref{cor4.10} and Proposition \ref{prop4.9} yield the following.

\begin{cor}
Let $X$ be a block subspace of $\X$. Then for every $n\inn$ there exists $x\in X$ and $f\in W$, such that $\{x,f\}$ is an $(n,1)$-exact pair.\label{corexactpairexist}
\end{cor}

We remind that, as Propositions \ref{prop3.5} and \ref{cor4.7}
state, if a sequence generates an $\ell_1$ spreading model, then
passing, if necessary, to a subsequence, it generates an
$\ell_1^k$ spreading model for any $k\inn$. However, as the next
proposition states, the space $\X$ does not admit higher order
$c_0$ spreading models.

\begin{prp} The space $\X$ does not admit $c_0^2$ spreading
models.\label{rem4.11}
\end{prp}

\begin{proof}
Towards a contradiction, assume that there is a sequence
$\{x_k\}_{k\inn}$ in $\X$, generating a $c_0^2$ spreading model.
Then it must be weakly null and we may assume that it is a
normalized block sequence. By Proposition \ref{prop4.9}, it
follows that there exist $\{F_k\}_{k\inn}$ increasing, Schreier
admissible subsets of the naturals and $c>0$ such that
$\|\sum_{j=1}^n\sum_{i\in F_{k_j}}x_i\| \geqslant n\cdot c$ for
any $n\leqslant k_1 <\ldots< k_n$. Since for any such
$F_{k_1}<\cdots< F_{k_n}$ we have that $\cup_{j=1}^n
F_{k_j}\in\mathcal{S}_2$, it follows that $\{x_k\}_{k\inn}$ does
not generate a $c_0^2$ spreading model.

\end{proof}

\begin{cor}Let $Y$ be an infinite dimensional closed subspace of $\X$. Then
$Y^*$ admits a spreading model equivalent to $\ell_1$ as well as a
spreading model equivalent to $c_0^n$, for any
$n\inn$.\label{cor4.12}
\end{cor}

\begin{proof}
Since $Y$ contains a sequence $\{x_k\}_{k\inn}$ generating a
spreading model equivalent to $c_0$, which we may assume is
Schauder basic, then for any normalized $\{x_k^*\}_{k\inn}\subset
Y^*$, such that $x^*_k(x_m) = \de_{n,m}$ for $n,m\in\mathbb{N}$,
we have that passing, if necessary, to a subsequence,
$\{x_k^*\}_{k\inn}$ generates a spreading model equivalent to
$\ell_1$.

To see that $Y^*$ admits a spreading model equivalent to $c_0^n$
for any $n\inn$, take the previously used sequence
$\{x_k\}_{k\inn}$. Working just like in the proof of Proposition
\ref{prop4.9} find $\{F_k\}_{k\inn}$ successive subsets of the
natural such that $\min F_k\geqslant \#F_k$, for all $k\inn$, if
$y_k = \sum_{i\in F_k}x_i$ for all $k\inn$, then $\{y_k\}_{k\inn}$
is seminormalized and there exists a very fast growing sequence of
$\al$-averages $\{\al_k\}_{k\inn}\subset W$ such that
$\lim\inf\al_k(\sum_{i\in F_k}x_i)\geqslant 1$.

Then, if $c = \lim\sup_k\|y_k\|$,we evidently have that
$\lim\inf_k\|\al_k\|\geqslant 1/c$ and since for any $n\inn,
F\in\mathcal{S}_n$, we have that $\frac{1}{2^n}\sum_{q\in
F}\al_{q}$ is a functional of type I$_\al$ in $W$, it follows that
$\|\sum_{q\in F}^d\al_{q}\|\leqslant 2^n$. This means that,
$\{\al_k\}_{k\inn}$ generates a spreading model equivalent to
$c_0^n$, with an upper constant $2^n$.

Let $I^*:\X^*\rightarrow Y^*$ be the dual operator of
$I:Y\rightarrow\X$. Then $\{I^*(\al_k)\}_{k\inn}$ generates a
spreading model equivalent to $c_0^n$, for any $n\inn$. Since
$\|I^*\| = 1$, all that needs to be shown is that
$\lim\inf_k\|I^*(\al_k)\|
> 0$. Indeed,
\begin{equation*}
\liminf_k\|I^*(\al_k)\|\geqslant
\liminf_k(I^*\al_k)(\frac{\sum_{i\in F_k}x_i}{c}) =
\liminf_k\al_k(\frac{\sum_{i\in F_k}x_i}{c})\geqslant 1/c
\end{equation*}

\end{proof}

\section{Properties of $\X$ and $\mathcal{L}(\X)$}

In this final section it is proved that $\X$ is hereditarily
indecomposable and the properties of the operators acting on
infinite dimensional closed subspaces of $\X$ are presented.

\subsection*{Dependent sequences and the HI property of $\X$.}
In the first part of this subsection we introduce the dependent
sequences, which are the main tool for proving the HI property of
$\X$ and studying the structure of the operators.

\begin{dfn} A sequence of pairs $\{x_k, f_k\}_{k=1}^n$, is said to be a
1-dependent sequence (respectively a 0-dependent sequence) if the
following are satisfied.
\begin{enumerate}

\item[(i)] $\{x_k, f_k\}$ is an $(m_k,1)$-exact pair (respectively an $(m_k,0)$-exact pair) for $k=1,\ldots,n$, with $m_1 > 4n2^{2n}$

\item[(ii)] $\max\supp f_k < \min\supp x_{k+1}$ for $k=1,\ldots,n-1$

\item[(iii)] $\{f_k\}_{k=1}^n$ is an $\mathcal{S}_1$-admissible
special sequence of type I$_\al$ functionals in $W$, i.e. $f = \frac{1}{2}\sum_{k=1}^nf_k$ is a functional of type II in $W$.

\end{enumerate}
\label{defdependent}
\end{dfn}

\begin{prp} Let $X$ be a block subspace of $\X$ and $n\inn$. Then there exist
$x_1,\ldots,x_n$ in $X$ and $f_1,\ldots,f_n$ in $W$, such that $\{x_k, f_k\}_{k=1}^n$ is a 1-dependent sequence.
\label{dependentsequenceexist}
\end{prp}

\begin{proof}
Choose $m_1\in L_1$ with $m_1 > 4n2^{2n}$. By Corollary \ref{corexactpairexist} there
exists $\{x_1, f_1\}$ an $(m_1,1)$-exact pair in $X$. Then $\min\supp f_1 \geqslant \min\supp x_1 > n$.

Let $d<n$ and suppose that we have chosen $\{x_k, f_k\}$ $(m_k,1)$-exact
pairs for  $k=1,\ldots,d$ such that $\{f_k\}_{k=1}^m$ is a
special sequence and $\max\supp f_k < \min\supp x_{k+1}$ for
$k=1,\ldots,d$.

Set $m_{d+1} = \sigma\big((f_1,m_1),\ldots,(f_d,m_d)\big)$. Then applying
Corollary \ref{corexactpairexist} once more, there exists
$\{x_{d+1},f_{d+1}\}$ an $m_{d+1}$-exact pair in
$X$, such that $\max\supp f_d < \min\supp x_{d+1}$.

The inductive construction is complete and
$\{x_k,f_k,\}_{k=1}^n$ is a 1-dependent sequence.

\end{proof}

An easy modification of the above proof yields the following.

\begin{cor}
If $X, Y$ are block subspaces of $\X$ and $n\inn$, then a 1-dependent sequence
$\{x_k,f_k,\}_{k=1}^{2n}$ can be chosen, such that
$x_{2k-1}\in X$ and $x_{2k}\in Y$ for $k=1,\ldots,n$.\label{remarkforhi}
\end{cor}

\begin{prp} Let $\{(x_k,f_k)\}_{k=1}^{2n}$ be a 1-dependent sequence in $\X$ and
set $y_k = x_{2k-1} - x_{2k}$, for $k=1,\ldots,n$. Then we have
that:
\begin{enumerate}

\item[(i)] $\frac{1}{n}\|\sum_{k=1}^{2n}x_k\| \geqslant 1$

\item[(ii)] $\frac{1}{n}\|\sum_{k=1}^ny_k\| \leqslant
\frac{232}{n}$
\end{enumerate}\label{prop5.2}
\end{prp}

\begin{proof}
Since $\frac{1}{2}\sum_{k=1}^{2n}f_k$ is a type II functional in
$W$, it immediately follows that
$\frac{1}{n}\|\sum_{k=1}^{2n}x_k\| \geqslant
\frac{1}{2n}\sum_{k=1}^{2n}f_k(x_k) = 1$.

By Remark \ref{remestexactpair} it follows that $1 \leqslant \|y_k\| \leqslant 58$, for $k=1,\ldots,n$.

Set $y_k^\prime = \frac{1}{58}y_k$ and $j_k = m_{2k-1}-2$ for $k=1,\ldots,n$. We will show that the assumptions of Proposition \ref{prop4.6} are satisfied. From this, it will follow
that $\frac{1}{n}\|\sum_{k=1}^ny_k\| \leqslant 58\frac{4}{n}$,
which is the desired result.

The first and second assumptions, follow from Lemmas \ref{lem3.7}
and \ref{lem3.15} respectively and the definition of the
1-dependent sequence.

The third assumption follows from the fact that, by the definition
of the 1-dependent sequence, $\max\supp f_k > \max\supp x_k$, for
$k=1,\ldots,2n$ and the definition of the coding function
$\sigma$.

It remains to be proven that the fourth assumption is also
satisfied. Let $1\leqslant k_1<k_2<k_3\leqslant n$ and $g =
\frac{1}{2}\sum_{j=1}^dg_j$ be a functional of type II that
separates $y_{k_1}^\prime,y_{k_2}^\prime$ and $y_{k_3}^\prime$.

Set $j_0 = \min\{j: \ran g_j\cap\ran y_{k_3}^\prime\neq\varnothing\}$ and
assume first that $w(f_{j_0}) = m_{2k_3-1}$ Since $\supp
g\cap\supp y_{k_1}\neq\varnothing$, it follows that $g_{j_0-1} =
f_{2k_3-2}$ and there exists $I$ an interval of the naturals,
$\ran y_{k_2}^\prime\subset I$, such that $g =
I(\frac{1}{2}\sum_{k=1}^{j_0 - 1}f_k)$. This yields that
$g(y_{k_2}^\prime) = 0$.

Otherwise, if $w(f_{j_0}) \neq m_{2k_3-1}$, set $g^\prime =
g|_{\ran y_{k_3}^\prime}$ and Corollary \ref{corx3.14} yields the
following.
\begin{equation*}
|g^\prime(y_{k_3}^\prime)| <
\frac{2\cdot4}{58}\frac{29}{28}\big(\frac{1}{2^{m_{2k_3-1}}} +
\frac{1}{2^{2m_{2k_3-1}}} +
\sum_{\substack{j\in\widehat{w}(g^\prime):\\w(g_j)<n}}\frac{4}{2^{w(g_j)}}\big)
\end{equation*}
Since $g$ separates $y_{k_1}, y_{k_2}$ and $y_{k_3}$, we have that
$\min\widehat{w}(g^\prime) > p_0 = \min\supp x_1$,
therefore
\begin{equation*}
\sum_{\substack{j\in\widehat{w}(g^\prime):\\w(g_j)<n}}\frac{1}{2^{w(g_j)}}
< \sum_{p > p_0}\frac{1}{2^p} = \frac{1}{2^{p_0}} \leqslant \frac{1}{2^{32\cdot2^{2m_1}}}
\end{equation*}
By the choice of $m_1$, we conclude that $|g(y_{k_3}^\prime)| < \frac{1}{n2^{n}}$, which means that the fourth assumption is
satisfied.

\end{proof}

The next proposition is proved by using similar arguments.

\begin{prp} Let $\{(x_k,f_k)\}_{k=1}^n$ be a 0-dependent sequence in $\X$.
Then we have that:
\begin{equation*}
\frac{1}{n}\|\sum_{k=1}^nx_k\| \leqslant \frac{112}{n}
\end{equation*}\label{prop5.3}
\end{prp}

We pass to the main structural property of $\X$.

\begin{thm} The space $\X$ is hereditarily indecomposable.\label{cor5.4}
\end{thm}

\begin{proof}
It is enough to show that for $X, Y$ block subspaces of $\X$, for
any $\e>0$, there exist $x\in X$ and $y\in Y$, such that
$\|x+y\|\geqslant 1$ and $\|x-y\| < \e$. Let $n\inn,$ such that
$\frac{232}{n}<\e$.

By Corollary \ref{remarkforhi} there exists a 1-dependent sequence
$\{x_k,f_k,\}_{k=1}^{2n}$, such that
$x_{2k-1}\in X$ and $x_{2k}\in Y$ for $k=1,\ldots,n$.

Set $x = \frac{1}{n}\sum_{k=1}^nx_{2k-1}$ and $y =
\frac{1}{n}\sum_{k=1}^nx_{2k}$. By applying Proposition
\ref{prop5.2}, the result follows.

\end{proof}

\subsection*{The structure of $\mathcal{L}(Y,\X)$}
For $Y$ a closed subspace of $\X$ and $T:Y\rightarrow\X$ we show
that $T = \la I_{Y,\X} + S$ with $S:Y\rightarrow \X$ strictly
singular.

\begin{prp} Let $Y$ be a subspace of
$\X$ and $T:Y\rightarrow \X$ be a linear operator, such that there
exists $\{x_k\}_{k\inn}$ a sequence in $Y$ generating a $c_0$
spreading model and $\lim\sup\dist(Tx_k,\mathbb{R}x_k) > 0$. Then
$T$ is unbounded.\label{prop5.5}
\end{prp}

\begin{proof}
Passing, if necessary, to a subsequence, there exists $1 > \de >
0$, such that $\dist(Tx_k,\mathbb{R}x_k) > \de$, for any $k\inn$.

Since $\{x_k\}_{k\inn}$ generates a $c_0$ spreading model, it is
weakly null. Set $y_k = Tx_k$ and assume that $T$ is bounded. It
follows that passing, if necessary, to a subsequence of
$\{x_k\}_{k\inn}$, then $\{y_k\}_{k\inn}$ also generates a $c_0$
spreading model.

We may assume that $\{x_k\}_{k\inn}$, as well as $\{y_k\}_{k\inn}$
are block sequences with rational coefficients. And $\lim_k\|x_k\|
= 1$, as well as $\lim_k\|y_k\| = 1$.

If this is not the case pass, if necessary, to a further
subsequence of $\{x_k\}_{k\inn}$, such that both $\{x_k\}_{k\inn}$
and $\{y_k\}_{k\inn}$ are equivalent to some block sequences with
rational coefficients $\{x_k^\prime\}_{k\inn},
\{y_k^\prime\}_{k\inn}$ respectively, and moreover
$\lim_k\|x_k^\prime\| = 1$, as well as $\lim_k\|y_k^\prime\| = 1$.
Set $Y^\prime = [\{x_k^\prime\}_{k\inn}]$ and
$T^\prime:Y^\prime\rightarrow\X$, such that $T^\prime x_k^\prime =
y_k^\prime$. It is easy to check that $T^\prime$ is also bounded
and $\dist(T^\prime x_k^\prime,\mathbb{R}x_k^\prime)
> \de^\prime$, for some $\de^\prime > 0$.

Set $I_k = \ran(\ran x_k\cup\ran y_k)$ and passing, if necessary,
to a subsequence, we have that $\{I_k\}_{k\inn}$ is an increasing
sequence of intervals of the naturals.

We will choose $\{f_k\}_{k\inn}\subset W$, such that $f_k(y_k) >
\frac{\de}{5}$, $f_k(x_k) = 0$ and $\ran f_k \subset I_k$, for all
$k\inn$.

The Hahn-Banach Theorem, yields that for all $k\inn$, there exists
$f^\prime_k\in B_{\X^*}$, such that $f^\prime_k(y_k) > \de$,
$f_k^\prime(x_k) = 0$ and $\ran f_k^\prime \subset I_k$, for all
$k\inn$.

By the fact that $\X$ is reflexive, it follows that $W$ is norm
dense in $B_{\X^*}$, therefore there exists $f_k^{\prime\prime}\in
W$ with $\|f_k^{\prime} - f_k^{\prime\prime}\| < \frac{\de}{4}$
and $\ran f_k^{\prime\prime}\subset I_k$, for all $k\inn$.

It follows that $f_k^{\prime\prime}(y_k) > \frac{3\de}{4},
|f_k^{\prime\prime}(x_k)| < \frac{\de}{4}$ and
$f_k^{\prime\prime}(x_k)$ is rational, for all $k\inn$.

Furthermore, there exists $g_k\in W$, such that $g_k(x_k) > 1 -
\frac{\de}{4}$, $g_k(x_k)$ is rational and $\ran g_k\subset I_k$,
for all $k\inn$.

Set $f_k = \frac{1}{2}(f_k^{\prime\prime} -
\frac{f_k^{\prime\prime}(x_k)}{g_k(x_k)}g_k)$. By doing some easy
calculations, it follows that the $f_k$ are the desired
functionals.

For the rest of the proof we may assume that the $\{x_k\}_k$ are normalized.

By copying the proof of Proposition \ref{prop4.9}, for any
$k_0\inn, n\inn$, there exists $F$ a finite subset of
the naturals with $\min F \geqslant k_0$ and $\{c_k\}_{k\in F}$
such that
\begin{enumerate}

\item[(i)] $z = 2^n\sum_{k\in F}c_kx_k$ is a $(4,1,n)$ exact vector

\item[(ii)] There exists a functional $f$ of type I$_\al$ with
weight $w(f) = n$ such that $f(z) = 0, \max\supp f > \max\supp z$
and if $w = 2^n\sum_{k\in F}c_ky_k$, then $f(w) > \frac{\de}{5}$.

\end{enumerate}

Using the above fact and arguing in the same way as in the proof of Proposition \ref{dependentsequenceexist},
for some $n\inn$, we construct a sequence $\{z_k\}_{k=1}^n$ and
$\{g_k\}_{k=1}^n$ such that $\{(z_k, g_k)\}_{k=1}^n$ is
0-dependent and if $w_k = Tz_k$, then $g_k(w_k) > \frac{\de}{5}$
and $\ran g_k\cap\ran w_m = \varnothing$, for $k\neq m$.

Then $f = \frac{1}{2}\sum_{k=1}^ng_k$ is a functional of type II
and $\frac{1}{n}\|\sum_{k=1}^nw_k\| \geqslant
\frac{1}{2n}\sum_{k=1}^ng_k(w_k) > \frac{\de}{10}$.

Moreover, Proposition \ref{prop5.3} yields that
$\frac{1}{n}\|\sum_{k=1}^nz_k\| \leqslant\frac{112}{n}$. It
follows that $\|T\| > \frac{n\cdot\de}{1120}$. Since $n$ was
randomly chosen, $T$ cannot be bounded, a contradiction which
completes the proof.

\end{proof}

In \cite{F}, it is proven that if $X$ is a hereditarily
indecomposable complex Banach space, $Y$ is a subspace of $X$ and
$T:Y\rightarrow X$ is a bounded linear operator, then there exists
$\la\in\mathbb{C}$, such that $T - \la I_{_{Y,X}}: Y \rightarrow
X$ is strictly singular. Here we prove a similar result for $\X$.

\begin{thm} Let $Y$ be an infinite dimensional closed subspace of
$\X$ and $T:Y\rightarrow \X$ be a bounded linear operator. Then
there exists $\la\in\mathbb{R}$, such that $T - \la I_{_{Y,\X}}: Y
\rightarrow \X$ is strictly singular.\label{prop5.6}
\end{thm}

\begin{proof}
If $T$ is strictly singular, then evidently $\la = 0$ is the
desired scalar.

Otherwise, choose $Z$ an infinite dimensional closed subspace of
$Y$, such that $T:Z\rightarrow\X$ is an into isomorphism. Choose
$\{x_k\}_{k\inn}$ a normalized sequence in $Z$ generating a $c_0$
spreading model. Proposition \ref{prop5.5} yields that
$\lim_k\dist(Tx_k,\mathbb{R}x_k) = 0$. Choose $\{\la_k\}_{k\inn}$
scalars, such that $\lim_k\|Tx_k - \la_k x_k\| = 0$ and $\la$ a
limit point of $\{\la_k\}_{k\inn}$.

We will prove that $S = T - \la I_{_{Y,\X}}$ is strictly singular.
Towards a contradiction, suppose that this is not the case. Then
there exists $\{y_k\}_{k\inn}$ a normalized sequence in $Y$
generating a $c_0$ spreading model and $\de>0$, such that
$\|Sy_k\| = \|(T - \la I_{_{Y,\X}})y_k\| > \de$, for all $k\inn$.

As previously, we may assume that $\{x_k\}_{k\inn},
\{y_k\}_{k\inn}$ as well as $\{Sy_k\}_{k\inn}$ are all normalized
block sequences generating $c_0$ spreading models.

By Proposition \ref{prop5.5} and passing, if necessary, to a
subsequence, there exists $\mu\in\mathbb{R}$, such that
$\lim_k\|Sy_k - \mu y_k\| = 0$. Evidently $\mu\neq 0$, otherwise
we would have that $\lim_k\|Sy_k\| = 0$. Pass, if necessary, to a
further subsequence of $\{y_k\}_{k\inn}$, such that
$\sum_{k=1}^\infty\|Sy_k - \mu y_k\| < \frac{|\mu|}{232}$.

Observe that $\lim_k\|Sx_k\| = 0$ and therefore we may pass, if
necessary, to a subsequence of $\{x_k\}_{k\inn}$, such that
$\sum_{k=1}^\infty\|Sx_k\| < \frac{|\mu|}{232}$.

Arguing in the same manner as in the proof of Proposition
\ref{dependentsequenceexist}, for some $n\inn$ construct $\{z_k\}_{k=1}^{2n}$ and
$\{f_k\}_{k=1}^{2n}$ such that $z_{2k-1}$ is a linear combination
of $\{y_k\}_{k\inn}$, $z_{2k}$ is a linear combination of
$\{x_k\}_{k\inn}$ and $\{(z_k, f_k)\}_{k=1}^{2n}$ is a 1-dependent
sequence. Set $f = \frac{1}{2}\sum_{k=1}^{2n}f_k$, which is a
functional of type II in $W$.

If $w_k = z_{2k-1} - z_{2k}$, Proposition \ref{prop5.2} yields
that $\frac{1}{n}\|\sum_{k=1}^nw_k\| \leqslant \frac{232}{n}$.

On the other hand, we have that

\begin{eqnarray*}
\frac{1}{n}\|\sum_{k=1}^nSw_k\| &\geqslant&
\frac{1}{n}\big(\|\sum_{k=1}^nSz_{2k-1}\| -
\|\sum_{k=1}^nSz_{2k}\|\big)
\end{eqnarray*}
\begin{eqnarray*}
\phantom{a}&\geqslant& \frac{1}{n}\big(\|\sum_{k=1}^n\mu
z_{2k-1}\| -
\|\sum_{k=1}^n(Sz_{2k-1} - \mu z_{2k-1})\| - \frac{29|\mu|}{232}\big)\\
&\geqslant&
\frac{1}{n}\big(\frac{|\mu|}{2}\sum_{k=1}^nf_{2k-1}(z_{2k-1}) -
\frac{29|\mu|}{232} - \frac{29|\mu|}{232}\big)\\
&=& \frac{|\mu|}{2} - \frac{|\mu|}{4n} \geqslant \frac{|\mu|}{4}
\end{eqnarray*}

It follows that $\|S\| \geqslant \frac{n|\mu|}{928}$, where $n$
was randomly chosen. This means that $S$ is unbounded, a
contradiction completing the proof.

\end{proof}

\subsection*{Strictly Singular Operators}
In this subsection we study the action of strictly singular
operators on Schauder basic sequences in subspaces of $\X$.

\begin{prp} Let $Y$ be an infinite dimensional closed subspace of
$\X$ and $T:Y\rightarrow \X$ be a linear bounded operator. Then
the following assertions are equivalent.
\begin{enumerate}
\item[(i)] $T$ is not strictly singular.

\item[(ii)] There exists a sequence $\{x_k\}_{k\inn}$ in $Y$
generating a $c_0$ spreading model, such that $\{Tx_k\}_{k\inn}$
is not norm convergent to 0.
\end{enumerate}\label{prop5.7}
\end{prp}

\begin{proof}
Assume first that $T$ is not strictly singular and let $Z$ be an
infinite dimensional closed subspace of $Y$, such that $T|_Z$ is
an isomorphism. Since any subspace of $\X$ contains a sequence
generating a $c_0$ spreading model, then so does $Z$. Since $T|_Z$
is an isomorphism, the second assertion is true.

Assume now that there exists $\{x_k\}_{k\inn}$ a sequence in $Y$
generating a $c_0$ spreading model, such that $\{Tx_k\}_{k\inn}$
does not norm converge to 0. By Proposition \ref{prop5.5} and
passing, if necessary to a subsequence, there exists $\la\neq 0$,
such that $\lim_k\|Tx_k - \la x_k\| = 0$. Passing, if necessary,
to a further subsequence, we have that $\sum_{k=1}^\infty\|Tx_k -
\la x_k\| < \infty$. But this means that $\{x_k\}_{k\inn}$ is
equivalent to $\{Tx_k\}_{k\inn}$, therefore $T$ is not strictly
singular.

\end{proof}

\begin{dfn} Let $\{x_k\}_k$ be a normalized block sequence $\X$. We say that $\{x_k\}_k$ is of rank 1, if $\adx = 0$ and $\bdx = 0$.\label{defrank1}
\end{dfn}

\begin{dfn}
Let $\{x_k\}_k$ be a normalized block sequence in $\X$. We say that $\{x_k\}_k$ is of rank 2, if it satisfies one of the following.

\begin{itemize}

\item[(i)] $\al\big(\{x_k\}_k\big) = 0$ and for every $L$ infinite subset of the natural numbers, $\be\big(\{x_k\}_{k\in L}\big) > 0$

\item[(ii)] For every $L$ infinite subset of the natural numbers $\al\big(\{x_k\}_{k\in L}\big) > 0$ and for every $C\geqslant 1, \theta > 0, \{n_j\}_j$ strictly increasing sequence of natural numbers, $F_j \subset L$ and $\{c_k^j\}_{k\in F_j}, j\inn$ such that $w_j = 2^{n_j}\sum_{k\in F_j}c_k^jx_k$ are $(C,\theta,n_j)$ vectors for every $j\inn$, we have that $\be\big(\{w_j\}_j\big) = 0$

\end{itemize}
\label{defrank2}
\end{dfn}

\begin{dfn}

Let $\{x_k\}_k$ be a normalized block sequence. We say that $\{x_k\}_k$ is of rank 3, if for every $L$ infinite subset of the natural numbers, $\al\big(\{x_k\}_{k\in L}\big) > 0$ and there exist $C\geqslant 1, \theta > 0, \{n_j\}_j$ strictly increasing sequence of natural numbers, $F_j \subset L$ and $\{c_k^j\}_{k\in F_j}, j\inn$ such that $w_j = 2^{n_j}\sum_{k\in F_j}c_k^jx_k$ are $(C,\theta,n_j)$ are vectors for every $j\inn$, and $\be\big(\{w_j\}_j\big) > 0$
\label{defrank3}
\end{dfn}

\begin{rmk}
It follows easily from the definitions above, that every normalized block sequence has a subsequence that is of some rank. Moreover, if a normalized block sequence is of some rank, then any of its subsequences is of the same rank. We would also like to point out that we can neither prove nor disprove the existence of sequences of rank 3. The failure of the existence of such sequences, would yield that the composition of any two strictly singular operators defined on a subspace of $\X$, is a compact one.\label{rmkrankhereditary}
\end{rmk}

\begin{dfn}
Let $\{x_k\}_k$ be a weakly null sequence in $\X$. If it norm null, then we say that it is of rank 0. If it is seminormalized, we say that $\{x_k\}_k$ is of rank $i$, if there exists a normalized block sequence $\{x_k^\prime\}_k$ which is of rank $i$, such that $\sum_k\|\frac{x_k}{\|x_k\|} - x_k^\prime\| < \infty$.
\label{defrankweaklynull}
\end{dfn}

\begin{rmk}
Every weakly null sequence in $\{x_k\}_k$ has a subsequence which is of some rank. Moreover Propositions \ref{prop3.5} and \ref{cor4.7} yield that $\{x_k\}_k$ has a subsequence which is of rank 1 if and only it admits a $c_0$ spreading model and it has a subsequence that is of rank 2 or 3 if and only if it admits $\ell_1$ as a spreading model.\label{rmksmrankrelation}
\end{rmk}

Proposition \ref{prop5.7} yields the following.

\begin{prp}
Let $Y$ be an infinite dimensional closed subspace of $\X$ and $T:Y\rightarrow \X$ be a strictly singular operator. Then for every $\{x_k\}_k$ weakly null sequence in $Y$ which is of rank 1, we have that $\{Tx_k\}_k$ is of rank 0.\label{prp1to0}
\end{prp}

\begin{prp}
Let $Y$ be an infinite dimensional closed subspace of $\X$ and $T:Y\rightarrow \X$ be a strictly singular operator. Then for every $\{x_k\}_k$ weakly null sequence in $Y$ which is of rank 2, we have that $\{Tx_k\}_k$ has no subsequence which is of rank 2 or of rank 3.\label{rank2to1}
\end{prp}

\begin{proof}
Towards a contradiction pass to a subsequence of $\{x_k\}_k$ and assume that there exist $\{x_k^\prime\}_k$, $\{y_k\}_k$ normalized block sequences, with $\sum_k\|\frac{x_k}{\|x_k\|} - x_k^\prime\| < \infty, \sum_k\|\frac{Tx_k}{\|Tx_k\|} - y_k\| < \infty$, $\{x_k^\prime\}_k$ satisfies either (i) or (ii) from Definition \ref{defrank2} and $\{y_k\}_k$ is of either rank 2 or 3. By Remark \ref{rmksmrankrelation}, we may assume that both $\{x_k^\prime\}_k$, $\{y_k\}_k$ generate $\ell_1$ as a spreading model.

Setting $T^\prime:[\{x_k^\prime\}_k]\rightarrow \X$ with $T^\prime x_k^\prime = y_k$ for all $k\inn$, we have that $T^\prime$ is bounded and strictly singular. Arguing in a similar manner as in the proof of Proposition \ref{prop3.5}, we may choose $\{n_j\}_j$ a  strictly increasing sequence of natural numbers, $\{F_j\}_j$ a strictly increasing sequence of natural numbers and $\{c_k^j\}_{k\in F_j}, j\inn$ such that $z_j = 2^{n_j}\sum_{k\in F_j}c_k^jx^\prime_k$ and $w_j = T^\prime z_j = 2^{n_j}\sum_{k\in F_j}c_k^jy_k$ are $(1,\theta,n_j)$ vectors for every $j\inn$. Proposition \ref{prop3.8} yields that $\al\big(\{z_j\}_j\big) = 0$ as well as $\al\big(\{w_j\}_j\big) = 0$.

If $\{x_k^\prime\}_k$ satisfies (i) from Definition \ref{defrank2}, then by Proposition \ref{prpalzeroalris} we may assume that is is $(2,\{m_k\}_k)$ $\al$-RIS, therefore the $z_j$ can in fact have been chosen to be $(2,\theta,n_j)$ exact vectors. Proposition \ref{cor4.7} yields that $\be\big(\{z_j\}_j\big) = 0$. We have concluded that $\{z_j\}_j$ is of rank 1 and by Proposition \ref{prp1to0} we have that $\{w_j\}_j$ is norm null, which contradicts the fact that $\|w_j\| \geqslant \theta$.

If on the other hand,  if $\{x_k^\prime\}_k$ satisfies (ii) from Definition \ref{defrank2}, then we have that $\be\big(\{z_j\}_j\big) = 0$. Again, Proposition \ref{prp1to0} yields that $\{w_j\}_j$ is norm null, which cannot be the case.
\end{proof}

\begin{prp}
Let $Y$ be an infinite dimensional closed subspace of $\X$ and $T:Y\rightarrow \X$ be a strictly singular operator. Then every $\{x_k\}_k$ weakly null sequence in $Y$, has a subsequence $\{x_{k_n}\}_n$ such that $\{Tx_{k_n}\}_n$ is of rank 0, of rank 1, or of rank 2.\label{rank3to2}
\end{prp}

\begin{proof}
If $\{Tx_k\}_k$ has a norm null subsequence then we are ok. Otherwise, pass to a subsequence of $\{x_k\}_k$, again denoted by $\{x_k\}_k$ and choose $\{x_k^\prime\}_k$, $\{y_k\}_k$ normalized block sequences, with $\sum_k\|\frac{x_k}{\|x_k\|} - x_k^\prime\| < \infty, \sum_k\|\frac{Tx_k}{\|Tx_k\|} - y_k\| < \infty$. By passing to a further subsequence and slightly perturbing the $x_k^\prime, y_k$, we may assume that $\min\supp x_k^\prime = \min\supp y_k$ for all $k\inn$ and that $\{x_k^\prime\}_k$, $\{y_k\}_k$ are of some rank.

Setting $T^\prime:[\{x_k^\prime\}_k]\rightarrow \X$ with $T^\prime x_k^\prime = y_k$ for all $k\inn$, we have that $T^\prime$ is bounded and strictly singular. Towards a contradiction, assume that $\{y_k\}_k$ satisfies the assumption of Definition \ref{defrank3}. By Remark \ref{rmksmrankrelation}, passing to a further subsequence, we have that $\{y_k\}_k$ generates an $\ell_1$ spreading model and by the boundedness of $T^\prime$, we may assume that so does $\{x_k^\prime\}_k$. Passing to an even further subsequence, we have that both $\{x_k^\prime\}_k$ and $\{y_k\}_k$ satisfy the conclusion of Proposition \ref{prop3.5}.

Choose $C\geqslant 1, \theta > 0, \{n_j\}_j$ strictly increasing sequence of natural numbers, $F_j \subset L$ and $\{c_k^j\}_{k\in F_j}, j\inn$ such that $w_j = 2^{n_j}\sum_{k\in F_j}c_k^jy_k$ are $(C,\theta,n_j)$ vectors for every $j\inn$, and $\be\big(\{w_j\}_j\big) > 0$.

Since $\min\supp x_k^\prime = \min\supp y_k$ for all $k\inn$, we have that $z_j = 2^{n_j}\sum_{k\in F_j}c_k^jx_k^\prime$ are $(C,\theta,n_j)$ vectors for every $j\inn$.

Proposition \ref{prop3.8} yields that $\al\big(\{z_j\}_j\big) = 0$ as well as $\al\big(\{w_j\}_j\big) = 0$.

Since $\be\big(\{w_j\}_j\big) > 0$, we may pass to a subsequence of $\{w_j\}_j$ that generates an $\ell_1$ spreading model and if $w_j^\prime = \frac{w_j}{\|w_j\|}$, then $\{w_j^\prime\}_j$ satisfies (i) from Definition \ref{defrank2}, it is therefore of rank 2.

Once more, the boundedness of $T^\prime$ yields that if $z_j^\prime = \frac{z_j}{\|w_j\|}$, then $\{z_j^\prime\}_j$ generates an $\ell_1$ spreading model. Since $\al\big(\{z_j^\prime\}_j\big) = 0$, we conclude that  $\be\big(\{z_j^\prime\}_j\big) > 0$. We may therefore pass to a final subsequence of $\{z_j^\prime\}_j$ which is of rank 2. Since $T^\prime z_j^\prime = w_j^\prime$, Proposition \ref{rank2to1} yields a contradiction.

\end{proof}

\subsection*{The Invariant Subspace Property}

\begin{thm} Let $Y$ be an infinite dimensional closed subspace of
$\X$ and $Q,S,T: Y\rightarrow Y$ be strictly singular operators.
Then $QST$ is compact.\label{cor5.9}
\end{thm}

\begin{proof}
Since $\X$ is reflexive, it is enough to show that for any weakly
null sequence $\{x_k\}_{k\inn}$, we have that $\{QSTx_k\}_{k\inn}$
norm converges to zero.

Proposition \ref{rank3to2}, yields that passing, if necessary to a subsequence, $\{Tx_k\}_k$ is of rank 0, rank 1, or rank 2. If it is of rank 0, then it is norm null and we are done. If it is or rank 1, Proposition \ref{prp1to0} yields that $\{STx_k\}_k$ is of rank 0 and as previously we are done. Otherwise, $\{Tx_k\}_k$ is of rank 2. By Proposition \ref{rank2to1}, we may pass to a further subsequence, such that $\{STx_k\}_k$ is either of rank 0, or rank 1. If it is not of rank 0, then applying Proposition \ref{prp1to0} we have that $\{QSTx_k\}_{k\inn}$
norm converges to zero and the proof is complete.
\end{proof}

\begin{cor}Let $Y$ be an infinite dimensional closed subspace of
$\X$ and $S:Y\rightarrow Y$ be a non zero strictly singular
operator. Then $S$ admits a non-trivial closed hyperinvariant
subspace.\label{cor5.10}
\end{cor}

\begin{proof}
Assume first that $S^3 = 0$. Then it is straightforward to check
that $\ker S$ is a non-trivial closed hyperinvariant subspace of
$S$.

Otherwise, if $S^3 \neq 0$, then Theorem \ref{cor5.9} yields that
$S^3$ is compact and non zero. Since $S$ commutes with its cube,
by Theorem 2.1 from \cite{Sir}, it is enough to check that for any
$\al,\be\in\mathbb{R}$ such that $\be\neq 0$, we have that $(\al I
- S)^2 + \be^2I \neq 0$. Since $S$ is strictly singular, it is
easy to see that this condition is satisfied.
\end{proof}

\begin{cor}Let $Y$ be an infinite dimensional closed subspace of
$\X$ and $T:Y\rightarrow Y$ be a non scalar operator. Then $T$
admits a non-trivial closed hyperinvariant
subspace.\label{cor5.11}
\end{cor}

\begin{proof}
Theorem \ref{prop5.6} yields that there exist $\la\in\mathbb{R}$,
such that $S = T - \la I$ is strictly singular, and since $T$ is
not a scalar operator, we evidently have that $S$ is not zero.

By Corollary \ref{cor5.10}, it follows that $S$ admits a
non-trivial closed hyperinvariant subspace $Z$. It is
straightforward to check that $Z$ also is a hyperinvariant
subspace for $T$.
\end{proof}

In the final result, which is related to Proposition 3.1 from
\cite{ADT}, we show that the ``scalar plus compact'' property
fails in every subspace of $\X$.

\begin{prp} Let $Y$ be an infinite dimensional closed subspace of
$\X$. Then there exists a strictly singular, non compact operator
$S:Y\rightarrow Y$. In fact, if $\mathcal{S}(Y)$ is the space of
strictly singular operators on $Y$, then $\mathcal{S}(Y)$ is
non-separable.\label{prop5.12}
\end{prp}

\begin{proof}
By Corollary \ref{cor4.10}, there exists a sequence
$\{x_k\}_{k\inn}$ in $Y$ that generates a spreading model
equivalent to $c_0$, say with an upper constant $c_1$ and by
Corollary \ref{cor4.12}, there exists a sequence
$\{x_k^*\}_{k\inn}$ in $Y^*$ that also generates a spreading model
equivalent to $c_0$, say with an upper constant $c_2$. Therefore
$\{x_k\}_{k\inn}$ and $\{x_k^*\}_{k\inn}$ are weakly null and we
may assume that they are Schauder basic and that $\dim(Y/[x_k]_k)
= \infty$. We may also assume that there exist $\{z_k\}_{k\inn}$
in $Y$ such that $\{x_k^*\}_{k\inn}$ is almost biorthogonal to
$\{z_k\}_{k\inn}$.

For $\e>0$, set $M_\e = \frac{4c_1}{\e}$. Choose a strictly
increasing sequence of naturals $\{q_j\}_{j\inn}$, such that
$q_j\geqslant M_{1/2^{j+1}}$. Set $S:Y\rightarrow Y$, such that
$Sx = \sum_{k=1}^\infty x_{q_k}^*(x)x_k$. Then:

\begin{enumerate}

\item[(i)] $S$ is bounded and non compact.

\item[(ii)] $S$ is strictly singular.

\end{enumerate}

We first prove that it is bounded. Let $x\in Y, \|x\| = 1$,
$x^*\in Y^*, \|x^*\| = 1$. For $j\geqslant 0$, set $B_j = \{k\inn:
1/2^{j+1}<|x^*(x_k)|\leqslant 1/2^j\}$. Since $\{x_k\}_{k\inn}$
generates $c_0$ as a spreading model, it follows that
$B_j\leqslant M_{1/2^{j+1}}\leqslant q_j$. Set $C_j = \{k\in B_j:
k\geqslant j\}, D_j = B_j\setminus C_j$. Evidently $\#D_j\leqslant
j$ and it is easy to see that $\#\{q_k: k\in
C_j\}\leqslant\min\{q_k: k\in C_j\}$, therefore, since
$\{x_k^*\}_{k\inn}$ generates a spreading model equivalent to
$c_0$, it follows that

\begin{equation*}
|\sum_{k\in C_j}x^*(x_k)x^*_{q_k}(x)| < c_2\max\{|x^*(x_k)|: k\in
C_j\}
\end{equation*}

Therefore $|\sum_{k\in B_j}x^*_{q_k}(x)x^*(x_k)|\leqslant
c_2\max\{|x^*(x_k)|: k\in C_j\} + j/2^j\leqslant c_2/2^j + j/2^j$.
From this it follows that
\begin{equation*}
\|Sx\|\leqslant \sum_{j=0}^\infty \frac{j + c_2}{2^j}\|x\|
\end{equation*}

The fact that $S$ is non compact follows easily if you consider
the almost biorthogonals $\{z_k\}_{k\inn}$ of
$\{x^*_{q_k}\}_{k\inn}$. Then $\{z_k\}_{k\inn}$ is a
seminormalized sequence in $Y$ and $\{Sz_k\}_{k\inn}$ does not
have a norm convergent subsequence.

We now prove that $S$ is strictly singular. Suppose that it is
not, then there exists $\la\neq 0$ such that $T = S - \la I$ is
strictly singular. Since $\la I$ is a Fredholm operator and $T$ is
strictly singular, it follows that $S = T + \la I$ is also a
Fredholm operator, therefore $\dim(Y/S[Y]) < \infty$. The fact
that $S[Y] \subset [x_k]_k$ and $\dim(Y/[x_k]_k) = \infty$ yields
a contradiction.

Moreover, for any further subsequence $\{x_k^*\}_{k\in L}$ of
$\{x_{q_k}^*\}_{k\inn}$, if we set $S_Lx = \sum_{k\in L}
x_{k}^*(x)x_k$, then $S_L$ satisfies the same conditions. This
yields that $\mathcal{S}(Y)$ contains an uncountable
$\e$-separated set and is therefore non-separable.
\end{proof}

The last proof actually yields that if $Y$ is an infinite
dimensional closed subspace of $\X$, then the space of strictly
singular, non-compact operators of $Y$ is non-separable.

\subsection*{Some final remarks}

We would like to mention that the structure of the dual of $\X$ is
unclear to us. In particular we cannot determine whether $\X^*$
shares similar properties with $\X$. For example, we do not know whether $\X^*$ admits only $c_0$ and $\ell_1$ as a spreading model. However, the following holds.

\begin{prp}
Let $X$ be a reflexive Banach space. Then the following are
equivalent.
\begin{itemize}

\item[(i)] The space $X$ satisfies the hereditary ISP.

\item[(ii)] Every infinite dimensional quotient of $X^*$ satisfies
ISP
\end{itemize}
If even more, every non scalar operator defined on an infinite
dimensional closed subspace of $X$ admits a non trivial closed
hyperinvariant subspace, then every non scalar operator defined on
an infinite dimensional quotient of $X^*$ admits a non trivial
closed hyperinvariant subspace.\label{dualityisp}
\end{prp}

\begin{proof}
In the general setting, if $X$ is a Banach space, $T$ is a bounded linear operator on $X$ admitting a non trivial closed invariant subspace $Y$, then it is straightforward to check that $Y^\perp$ is a non trivial closed invariant subspace of $T^*$. If moreover $Y$ is $T$-hyperinvariant, then $Y^\perp$ is $S^*$-invariant, for every operator $S$ on $X$, which commutes with $T$.

In the setting of reflexive spaces, all operators on $X^*$ are dual operators, hence we conclude the following. Let $X$ be a reflexive Banach space and $T$ be a bounded linear operator on $X$.

\begin{itemize}

\item[(a)] Then $T$ admits a non trivial closed invariant subspace if and only if $T^*$ admits a non trivial closed invariant subspace.

\item[(b)] Moreover, $T$ admits a non trivial closed hyperinvariant subspace if and only if $T^*$ admits a non trivial closed hyperinvariant subspace.

\end{itemize}

We now proceed to prove the equivalence of assertions (i) and (ii).

Assume that (i) holds, let $X^*/Y$ be an infinite dimensional quotient of $X^*$ and $T$ be a bounded linear operator on $X^*/Y$. Then $X^*/Y = Y_\perp^*$ and there is $S$ a bounded linear operator on $Y_\perp$ with $T = S^*$. Since $S$ admits a non trivial closed invariant subspace, by (a) so does $T$.

If we moreover assume that $S$ admits a non trivial closed hyperinvariant subspace, then by (b), so does $T$.

Conversely, if (ii) holds, assume that $Y$ is an infinite dimensional subspace of $X$ and $T$ is a bounded linear operator on $Y$. Then $Y^* = X^*/Y^\perp$. By the assumption, $T^*$ admits a non trivial invariant subspace and therefore, by (a) so does $T$.

\end{proof}

\begin{cor}
Every infinite dimensional quotient of $\X^*$ satisfies ISP. More
precisely, every non scalar operator defined on an infinite
dimensional quotient of $\X^*$ admits a non trivial closed
hyperinvariant subspace.
\end{cor}

We would also like to mention, that the method of constructing hereditarily indecomposable Banach spaces with saturation under constraints, using Tsirelson space as an unconditional frame, can be used to yield further results. For example, in \cite{AM} a reflexive hereditarily indecomposable Banach space $\mathfrak{X}_{_{^\text{usm}}}$ is constructed having the following property. In every
subspace $Y$ of $\mathfrak{X}_{_{^\text{usm}}}$ there exists a weakly null normalized
sequence $\{y_n\}_n$, such that every subsymmetric sequence
$\{z_n\}_n$ is isomorphically generated as a spreading model of a
subsequence of $\{y_n\}_n$.

\end{document}